\documentclass[11pt,a4paper]{article}
\usepackage{hyperref}
\usepackage{amsmath}
\usepackage{amsfonts}
\usepackage{amssymb}
\usepackage{amsthm}
\usepackage{mathrsfs}
\usepackage{mathtools}
 \usepackage[text={17cm,23cm},centering]{geometry}
\usepackage{cite}
\usepackage{graphicx}
\usepackage{enumitem}

\usepackage{color}
\usepackage[dvipsnames]{xcolor}

\newtheorem{thm}{Theorem}
\newtheorem{lem}[thm]{Lemma}
\newtheorem{prop}[thm]{Proposition}

\newtheorem*{thm*}{Theorem}
\theoremstyle{definition}
\newtheorem{defn}[thm]{Definition}

\theoremstyle{remark}

\DeclareMathOperator{\supp}{supp}
\DeclareMathOperator{\diam}{diam}
\DeclareMathOperator{\dist}{dist}

\DeclareMathOperator{\crit}{crit}
\DeclareMathOperator{\lip}{Lip}

\DeclareMathOperator{\acc}{acc}
\DeclareMathOperator{\essacc}{ess\,acc}
\DeclareMathOperator{\arclength}{arc\,length}
\DeclareMathOperator*{\argmin}{arg\,min}
\newcommand{\meas}[2]{\mu_{#1|_{#2}}}

\newcommand{\R}{\mathbb{R}}
\newcommand{\N}{\mathbb{N}}
\newcommand{\Z}{\mathbb{Z}}

\newcommand{\lebesgue}{\mathsf{Leb}}

\renewcommand{\geq}{\geqslant}
\renewcommand{\leq}{\leqslant}

\title{Examples of pathological dynamics of the subgradient method for Lipschitz path-differentiable functions}
\author{
 Rodolfo R\'ios-Zertuche
}

\begin{document}
\maketitle
\begin{abstract}
 We show that the vanishing stepsize subgradient method ---widely adopted for machine learning applications--- can display rather messy behavior even in the presence of favorable assumptions. 
 
 We establish that convergence of bounded subgradient sequences may fail even with a Whitney stratifiable objective function satisfying the Kurdyka-\L{}ojasiewicz inequality. 
 
 Moreover, when the objective function is path-differentiable we show that various properties all may fail to occur: criticality of the limit points, convergence of the sequence, convergence in values, codimension one of the accumulation set, equality of the accumulation and essential accumulation sets, connectedness of the essential accumulation set, spontaneous slowdown, oscillation compensation, and oscillation perpendicularity to the accumulation set.
\end{abstract}

\section{Introduction}
\label{sec:intro}

In our previous work \cite{ouroscillationcompensation}, we investigate the vanishing-step subgradient method applied to a nonsmooth, nonconvex objective function $f$ in the hope of finding
\[\argmin_{x\in\R^n}f(x).\]
This paper is intended as a companion to \cite{ouroscillationcompensation}, as it presents two examples that show that the results obtained there are sharp in several senses. We also aim here to provide insight into the types of dynamics that the subgradient algorithm presents in the asymptotic limit, and we evaluate some of the ideas that are believed to show promise towards a proof of convergence of the algorithm, such as the Kurdyka--\L{}ojasiewicz inequality. We refer the reader to \cite{ouroscillationcompensation} for some discussion of the historical background. 

\medskip
We shall now give some definitions that will allow us to discuss our results.

For a locally Lipschitz function $f\colon\R^n\to\R$, we denote by $\partial^cf(x)$ the \emph{Clarke subdifferential} of $f$ at $x\in\R^n$, that is, the convex envelope of the set of vectors $v\in\R^n$ such that there is a sequence $\{y_i\}_i\subset\R^n$ such that $f$ is differentiable at $y_i$, $y_i\to x$ and $\nabla f(y_i)\to v$.

\begin{defn}[Small-step subgradient method]\label{def:subgradientmethod}
 Let $f\colon\R^n\to\R$ be a locally Lipschitz function, and $\{\varepsilon_i\}_i$ be a sequence of positive step sizes such that 
 \[\sum_{i=0}^\infty\varepsilon_i=+\infty\qquad\textrm{and}\qquad \lim_{i\to+\infty}\varepsilon_i=0.\]
 Given $x_0\in\R^n$, consider the recursion, for $i\geq0$,
 \[x_{i+1}=x_i-\varepsilon_iv_i,\qquad v_i\in\partial^cf(x_i).\]
 Here, $v_i$ is chosen freely among $\partial^cf(x_i)$. The sequence $\{x_i\}_{i\in\N}$ is called a \emph{subgradient sequence}. %
\end{defn}

Since the dynamics of the subgradient method in the case of $f$ locally Lipschitz had been shown \cite{daniilidisdrusvyatskiy,borweinetal} to be too unwieldy, in \cite{ouroscillationcompensation} we instead discuss the dynamics of the subgradient method for $f$ path-differentiable.

\begin{defn}[Path-differentiable functions]\label{def:pathdiff}
 A locally Lipschitz function $f\colon\R^n\to\R$ is \emph{path-dif\-fer\-en\-tia\-ble} if for each Lipschitz\footnote{In other parts of the literature (see e.g. \cite{boltepauwels}), this definition is given with absolutely-continuous curves, and this is equivalent because such curves can be reparameterized (for example, by arclength) to obtain Lipschitz curves, without affecting their role in the definition.} curve $\gamma\colon\R\to\R^n$, for almost every $t\in\R$, the composition $f\circ\gamma$ is dif\-fer\-en\-tia\-ble at $t$ and the derivative is given by
 \[(f\circ\gamma)'(t)=v\cdot \gamma'(t)\]
 for all $v\in\partial^cf(\gamma(t))$.
\end{defn}
\begin{defn}[Weak Sard condition]\label{def:weaksard} 
 We will say that $f$ satisfies the \emph{weak Sard condition} if it is constant on each connected component of its critical set $\crit f=\{x\in\R^n:0\in\partial^c f(x)\}$.
\end{defn}

Recall that the \emph{accumulation set $\acc\{x_i\}_i$} of the sequence $\{x_i\}_i$ is the set of points $x\in\R^n$ such that, for every neighborhood $U$ of $x$, the intersection $U\cap\{x_i\}_i$ is an infinite set. Its elements are known as \emph{limit points}.

\begin{defn}[Essential accumulation set]\label{def:essacc}
 Given sequences $\{x_i\}_i\subset\R^n$ and $\{\varepsilon_i\}_i\subset\R_{\geq0}$, the \emph{essential accumulation set $\essacc\{x_i\}_i$} is the set of points $x\in\R^n$ such that, for every neighborhood $U$ of $x$,
 \begin{equation}\label{eq:essaccdef}
  \limsup_{N\to+\infty}\frac{\displaystyle\sum_{\substack{0\leq i\leq N\\ x_i\in U}}\varepsilon_i}{\displaystyle\sum_{0\leq i\leq N}\varepsilon_i}>0.\end{equation}
\end{defn}

\begin{defn}[Whitney stratifiable functions]\label{def:stratification}
 Let $X$ be a nonempty subset of $\R^m$ and $0<p\leq +\infty$. A \emph{$C^p$ stratification} $\mathcal X=\{X_i\}_{i\in I}$ of $X$ is a locally finite partition of $X=\bigsqcup_iX_i$ into connected submanifolds $X_i$ of $\R^m$ of class $C^p$ such that for each $i\neq j$
 \[\overline{X_i}\cap X_j\neq \emptyset \Longrightarrow X_j\subset \overline{X_i}\setminus X_i.\]
 A $C^p$ stratification $\mathcal X$ of $X$ \emph{satisfies Whitney's condition A} if, for each $x\in \overline{X_i}\cap X_j$, $i\neq j$, and for each sequence $\{x_k\}_k\subset X_i$ with $x_k\to x$ as $k\to+\infty$, and such that the sequence of tangent spaces $\{T_{x_k}X_i\}_k$ converges (in the usual metric topology of the Grassmanian) to a subspace $V\subset T_x\R^m$, we have that $T_xX_j\subset V$. A $C^p$ stratification is \emph{Whitney} if it satisfies Whitney's condition A.

 With the same notations as above,
 a function $f\colon \R^n\to\R^k$ is \emph{Whitney $C^p$-stratifiable} if there exists a Whitney $C^p$ stratification of its graph as a subset of $\R^{n+k}$. 
\end{defn}

\paragraph{Summary of the results.}

Let
\begin{itemize}
 \item $n>0$, 
 \item $f\colon\R^n\to\R$ be a locally Lipschitz, path-differentiable function, 
 \item the sequence $\{\varepsilon_{i}\}_i\subset\R_{>0}$ of step sizes satisfy $\lim_{i\to+\infty}\varepsilon_i=0$, and
 \item $\{x_i\}_i$ be a bounded subgradient sequence with stepsizes $\{\varepsilon_i\}_i$.
\end{itemize}
The main questions we address here are the following:

\begin{enumerate}[label=Q\arabic*.,ref=Q\arabic*,series=Q]
 \item\label{q:1}\label{q:prevfirst} \emph{Does the sequence $\{x_i\}_i$ converge  in general?} 
 
 While it is tempting to hope for the sequence to converge since we have proven \cite[Theorems 6(i),7(i),7(ii)]{ouroscillationcompensation} that the sequence slows down indefinitely, in Section \ref{sec:circle} we give an example in which the sequence forever accumulates around a circle and never converges. The function we construct satisfies the weak Sard condition, so even with that assumption there is no hope for the convergence of $\{x_i\}_i$. The function also satisfies the Kurdyka--\L{}ojasiewicz inequality; see \ref{q:9}. 
 
 In contrast, in can be proven \cite{bolte2010characterizations} that if $f$ satisfies the weak Sard condition and the Kurdyka--\L{}ojasiewicz inequality, then the flow lines $x\colon\R\to\R^n$ of the continuous-time subgradient flow, which satisfy
 \[-\dot x(t)\in\partial^cf(x(t)),\]
 always converge. Thus the example in Section \ref{sec:circle} shows that the convergence of the continuous-time process may not guarantee the convergence of the discrete subgradient sequence.
 
 \item\label{q:5} \emph{Do the values $\{f(x_i)\}_i$ converge for a general path-differentiable function $f$?} 
 
 Although this convergence can be proved when $f$ satisfies the weak Sard condition \cite[Theorem 7(v)]{ouroscillationcompensation}, the example in Section \ref{sec:sardcounterexample} shows that the convergence of the values $f(x_i)$ fails in general. In fact, in that example we have $f(\acc\{x_i\}_i)=[0,1]=f(\essacc\{x_i\}_i)$.

 \item\label{q:2} \emph{Must $\acc\{x_i\}_i$ be a subset of $\crit f$ in general?}
 
 The example in Section \ref{sec:sardcounterexample} shows that in general the set $\acc\{x_i\}_i\setminus\essacc\{x_i\}_i$ may not intersect $\crit f$. This contrasts with results that $\essacc\{x_i\}_i$ is always contained in $\crit f$ \cite[Theorem 6(iii)]{ouroscillationcompensation}, and that $\acc\{x_i\}_i$ is contained in $\crit f$ if $f$ satisfies the weak Sard condition \cite[Theorem 7(iv)]{ouroscillationcompensation}.
 
  \item\label{q:4} \emph{Do we always have $\essacc\{x_i\}_i=\acc\{x_i\}_i$?} 
 
 No, in the example in Section \ref{sec:sardcounterexample} we have a situation in which the set $\essacc\{x_i\}_i$ is strictly smaller than $\acc\{x_i\}_i$. We do not know the answer to this question with more stringent assumptions, such as $f$ satisfying the weak Sard condition.

 \item\label{q:6}\emph{Can the essential accumulation set $\essacc\{x_i\}_i$ be disconnected?} 
 
 Yes. Although for simplicity we do not construct an example here, the reader will surely understand that the example in Section \ref{sec:sardcounterexample} can be easily modified (by taking several copies of $\Gamma$ and joining them with curves having roles similar to the one played by $J$) to produce a situation in which $\essacc\{x_i\}_i$ is disconnected. This contrasts with the fact that $\acc\{x_i\}_i$ is always connected because $\dist(x_i,x_{i+1})\leq \varepsilon_i\lip(f)\to 0$ as $i\to+\infty$, where $\lip(f)$ is the Lipschitz constant for $f$ in a compact set that contains $\{x_i\}_i$.
 
 \item\label{q:3} \emph{A certain spontaneous slowdown phenomenon is proved in \cite[Theorem 6(i)]{ouroscillationcompensation} of the fragments of the subgradient sequence as (roughly speaking) it traverses the piece of $\acc\{x_i\}_i$ starting at a point $x$ and ending at another point $y$, such that $x,y\in\acc\{x_i\}$ verify $f(x)\leq f(y)$ (see the precise statement below). }
 
 \emph{Is there any hope of proving, for general $f$, 
 that this phenomenon always occurs uniformly throughout the accumulation set, regardless of the restriction $f(x)\leq f(y)$? }

 No, the example in Section \ref{sec:sardcounterexample} shows that the speed of drift of the sequence can remain high forever between points that do not satisfy this inequality.
 
 To be precise, the result in \cite[Theorem 6(i)]{ouroscillationcompensation} is this: Let $x$ and $y$ be two distinct points in $\acc\{x_i\}_i$ satisfy $f(x)\leq f(y)$, and take subsequences $\{x_{i_k}\}_k$ and $\{x_{i'_k}\}_k$ such that $x_{i_k}\to x$, $x_{i'_k}\to y$ as $k\to+\infty$, and $i'_k>i_k$ for all $k$. Then 
 \[\lim_{k\to+\infty}\sum_{p=i_k}^{i'_k}\varepsilon_p=+\infty.\]
 This is verified independently of the subsequences taken.
 
 On the other hand, the endpoints $x$ and $y$ of the curve $J$ in the example in Section \ref{sec:sardcounterexample} are contained in $\acc\{x_i\}_i$, satisfy $f(x)>f(y)$, and we can take subsequences $\{x_{i_k}\}_k$ and $\{x_{i'_k}\}_k$ converging to $x$ and $y$, respectively, and with $i'_k>i_k$,
 for which we additionally have
 \[\sup_{k}\sum_{p=i_k}^{i'_k}\varepsilon_p <+\infty.\]

 \item\label{q:7} \emph{Does the oscillation compensation phenomenon described in \cite[Theorem 6(ii)]{ouroscillationcompensation} occur on the entire accumulation set in general?} 
 
 While we are able to prove an oscillation compensation result \cite[Theorem 7(iii)]{ouroscillationcompensation} that holds throughout $\acc\{x_i\}_i$ with the assumption that $f$ satisfies the weak Sard condition, the example in Section \ref{sec:sardcounterexample} shows that in general, in the absence of the weak Sard condition, there need not be any oscillation compensation on $\acc\{x_i\}_i\setminus\essacc\{x_i\}_i$, which in the example corresponds to the curve $J$. For a precise statement, please refer to \ref{itex:osccomp} in Section \ref{sec:sardcounterexample}.
 
 \item\label{q:8} \emph{Can the perpendicularity of the oscillations of $\{x_i\}_i$ verified around $\essacc\{x_i\}_i$ \cite[Remark 9]{ouroscillationcompensation} be proved on the entire accumulation set?} 
 
 No, as is shown in the example of Section \ref{sec:sardcounterexample} this may fail on $\acc\{x_i\}_i\setminus\essacc\{x_i\}_i$ for general $f$. The perpendicularity can, however, be proved to happen on $\essacc\{x_i\}_i$ or, if $f$ satisfies the weak Sard condition, on all of $\acc\{x_i\}_i$; see \cite[Remark 9]{ouroscillationcompensation}. 
 
 \item\label{q:9} \emph{Would it be possible to prove the convergence of $\{x_i\}_i$ if $f$ is Whitney stratifiable (cf. Definition \ref{def:stratification}) and satisifes a Kurdyka--\L{}o\-ja\-sie\-wicz inequality?} 
 
 No; more assumptions are necessary. The objective function $f$ in the example in Section \ref{sec:circle} is Whitney $C^\infty$ stratifiable and satisfies a Kurdyka--\L{}ojasiewicz inequality of the form
 \[\|\nabla f(x)\|\geq \frac12\quad\textrm{for all $x\notin \crit f$,}\]
 but we also construct a bounded subgradient sequence that fails to converge. However, in the case of $f$ smooth, the Kurdyka--\L{}ojasiewicz inequality does suffice to prove convergence of the subgradient method \cite{attouchboltesvaiter}.
 \item\label{q:10}  Recall that the Hausdorff dimension of a set $X$ is 
 \[\dim X=\inf\{d\in\R:\mathcal H^d(X)=0\},       \]
 where $\mathcal H^d(X)$ is the $d$-dimensional Hausdorff outer measure,
 \begin{equation*}
   \mathcal H^d(X)\coloneqq\liminf_{r\to0}\{\textstyle\sum_ir_i^d:
   \textrm{there is a cover of $X$ by balls of radii $0<r_i<r$}\}.
 \end{equation*}
 \emph{Must the Hausdorff dimension of the accumulation set of $\{x_i\}_i$ be $\dim\acc\{x_i\}_i\leq n-1$?} 
 
 No, the example in Section \ref{sec:sardcounterexample} gives a function $f\colon\R^2\to\R$ and a subgradient sequence $\{x_i\}_i$ such that the Hausdorff dimension satisfies
 \begin{equation}\label{eq:fractaldimension}
  1<\dim\acc\{x_i\}_i=\dim\essacc\{x_i\}_i\leq \frac{\log 4}{\log 3}\approx 1.26,
 \end{equation}
 and actually depends on a parameter $\alpha$ that can be tweaked to produce any value of the Hausdorff dimension in this range; see Lemma \ref{lem:dimgamma}.
 Although the function $f$ in that example does not satisfy the weak Sard condition, the example can be easily modified (by changing the value of $f$ on $\Gamma\cup J$ to a constant) to satisfy also this condition and still have the dimension attain any value in the range \eqref{eq:fractaldimension}.
 
 This contrasts with the result \cite[Remark 10]{ouroscillationcompensation} that, if $f$ is Whitney $C^n$ stratifiable, then 
 \[\dim\acc\{x_i\}_i\leq n-1.\]
 
 \item\label{q:11}\label{q:prevlast}\emph{Can the set of limit closed measures of the interpolant curve be infinite?}
 
 Yes. This is the case in the situation of the example in Section \ref{sec:circle} (and also in the example of Section \ref{sec:sardcounterexample}, but for simplicity we will not prove it in that case). Please refer to Section \ref{sec:circlemeasures} for the full definitions and an explanation.
 
 \item\label{q:12}\emph{Would the answer to any of the previous questions \ref{q:prevfirst}--\ref{q:prevlast} be different if one enforced that
the sequence be contained in the (full measure) set of
differentiability points of the function $f$? }
 
  No, all our claims are based on
constructive existence proofs of subgradient sequences  $\{x_i\}_i$
such that each point $x_i$ is contained in a ball in which the objective function $f$ is $C^\infty$.
 \end{enumerate}

 \paragraph{Notation.}
 Given two sets $A$ and $B$, denote by $B^c$ the complement of $B$ and by $A\setminus B=A\cap B^c$.
 Let $n$ be a positive integer, and let $\R^n$ denote $n$-dimensional Euclidean space. %
 For two vectors $u=(u_1,\dots,u_n)$ and $v=(v_1,\dots,v_n)$ in $\R^n$, we let $u\cdot v=\sum_{i=1}^nu_iv_i$ and $\|u\|=\sqrt{u\cdot u}$. We will denote the gradient of $f$ at $x$ by $\nabla f(x)$. We denote $\log_ba=\log a/\log b$ the logarithm of $a$ in base $b$. We denote the unit circle by $S^1$, and the open ball of radius $r$ centered at $x$ by $B_r(x)$. A number with a subindex $b$ is in base $b$; for example, $0.12_9=1/9+2/81$. For a Lipschitz function $g\colon\R^n\to\R^m$, we denote by 
 \[\lip(g)=\sup_{x,y\in\R^n}\frac{\|g(x)-g(y)\|}{\|x-y\|}.\]

\section{Example on the circle}
\label{sec:circle}

 We construct a path-differentiable function $f\colon \R^2\to\R$ and a subgradient sequence $\{x_i\}_i$ that does not converge and instead accumulates around a circle. The function $f$ additionally has the property that it is Whitney $C^\infty$ stratifiable and satisfies a Kurdyka-\L{}ojasiewicz inequality. The construction is given in Section \ref{sec:circleconstruction} and the main properties are collected in Proposition \ref{prop:propertiesf}.
 
 In the context of the theory developed in \cite[\S4.2]{ouroscillationcompensation}, it is also interesting that the dynamics in this example induce, through the interpolant curve, infinitely-many limiting closed measures. This is discussed in Section \ref{sec:circlemeasures}.
 
 \subsection{Construction and main properties}\label{sec:circleconstruction}
 For $i\geq 2$, let (see Figure \ref{fig:example})
 \begin{figure}
  \centering
    \includegraphics[width=0.3\textwidth]{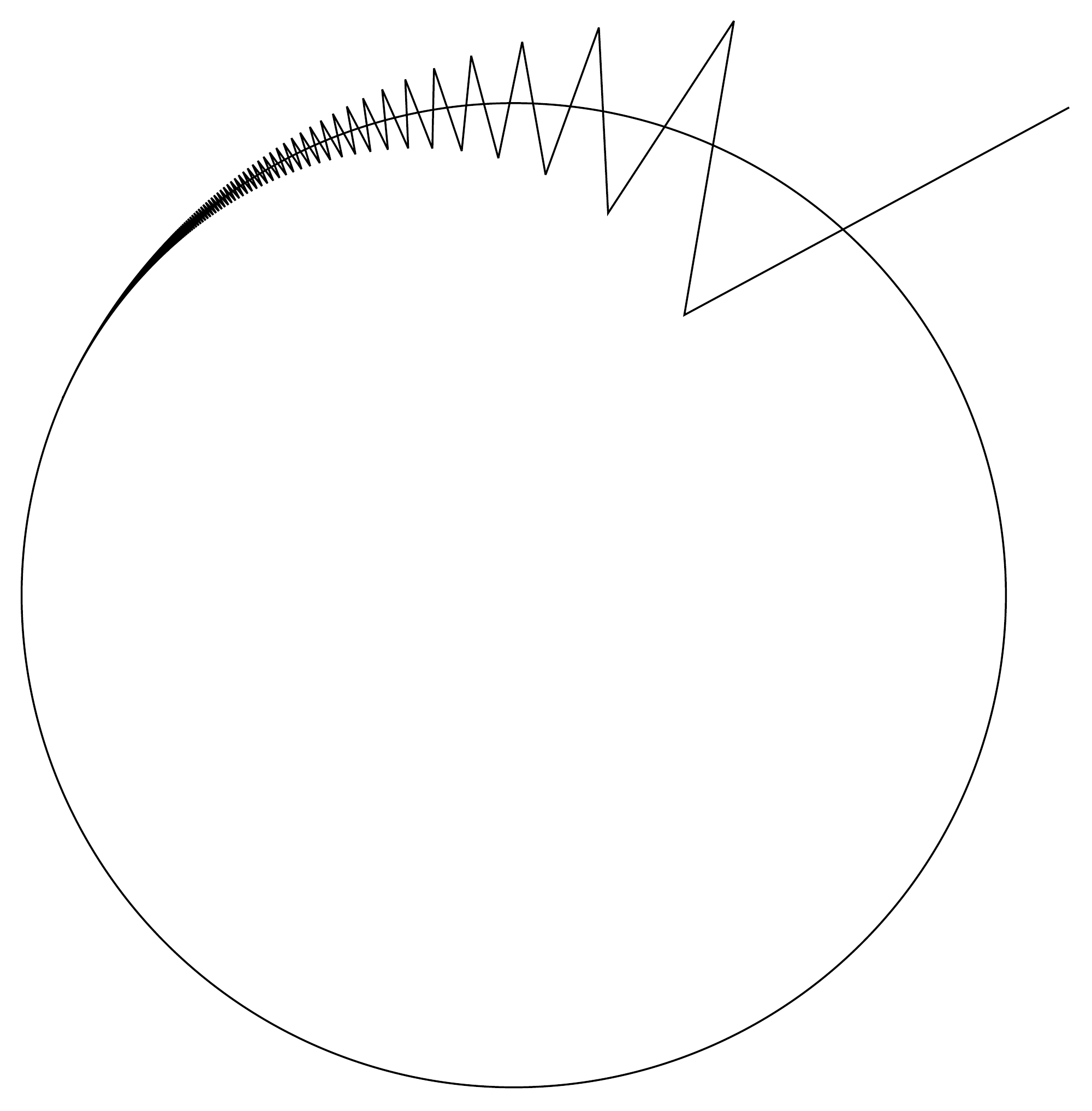}
  \caption{The unit circle and the path joining $\{x_i\}_i$ in the example of Section \ref{sec:circle}.}
  \label{fig:example}
\end{figure}
  \[x_i=[1+\tfrac{(-1)^i}i](\cos\vartheta_i,\sin\vartheta_i)\quad\textrm{with}\quad \vartheta_i=\sum_{k=2}^i\frac{1}{k\log k}\]
  and
  \[\varepsilon_i=\|x_{i+1}-x_i\|,\quad v_i=-\frac{x_{i+1}-x_i}{\|x_{i+1}-x_i\|},\]
so that $x_{i+1}=x_i-\varepsilon_iv_i$. Note that $\varepsilon_i$ satisfies, for large $i$,
\[%
 \frac{2}{i+1}<\varepsilon_i<\frac1i+\frac1{i+1}+\frac1{i\log i}< \frac2i,
\]
so that $\varepsilon_i\to0$, $\sum_i\varepsilon_i=+\infty$, and $\sum_i\varepsilon_i^2<+\infty$.

We want to obtain a function $f$ that is very close to the function $\phi$ given by the distance to the circle,
\[\phi(x)=  |1-\|x\||,\]
yet satisfies
 \begin{equation}\label{eq:wishes}
  \nabla f(x)= v_i\quad\textrm{for all $x\in B_{1/2^{i}}(x_i)$}.
 \end{equation}

 Let $\psi\colon\R^2\to[0,1]$ be a $C^\infty$ function with radial symmetry (i.e. $\psi(x)=\psi(y)$ for $\|x\|=\|y\|$), such that $\psi(x)=1$ for $x\in B_1(0)$, $\psi(x)=0$ for $\|x\|\geq 2$, and decreases monotonically on rays emanating from the origin. Let
 \[\psi_i(x)=\psi(2^i(x-x_i)),\]
 so that $\psi_i$ equals 1 on $B_{1/2^i}(x_i)$ and vanishes outside $B_{1/2^{i-1}}(x_i)$. %
 Note that the supports of the functions $\psi_i$ are pairwise disjoint.
 
 Define
 \[V_i(x)=(x-x_i)\cdot v_i+\frac1i.\]

 \begin{prop}\label{prop:propertiesf} Let $i_0\geq 2$ and
 \begin{equation}\label{eq:deff}
  f(x)=\left(1-\sum_{i=i_0}^\infty\psi_i(x)\right)\phi(x)+\sum_{i=i_0}^\infty\psi_i(x)V_i(x).
 \end{equation}
 Then we have:
  \begin{enumerate}[label=\roman*.,ref=(\roman*)]
   \item\label{f:regularity} The function $f$ is $C^\infty$ %
   on $\R^2\setminus S^1$. 
   \item\label{f:gradientsequence} The function $f$ satisfies \eqref{eq:wishes}, so that $\{x_i\}_i$ is a subgradient sequence with stepsizes $\{\varepsilon_i\}_i$.
   \item\label{f:clarke} Let $p$ be a point in the unit circle, then $\partial^c f(p)=\{ap:a\in[-1,1]\}=\partial^c\phi(p)$.
   \item\label{f:crit} The critical set of $f$ is $\crit f=S^1\cup \{0\}$.
   \item\label{f:pathdiff} The function $f$ is Lipschitz path-differentiable.
   \item\label{f:whitney} The function $f$ is Whitney $C^\infty$ stratifiable.
   \item\label{f:KL} If $i_0$ is large enough, $f$ satisfies a Kurdyka-\L{}ojasiewicz inequality of the form 
    \[\|\nabla f(x)\|>1/2\] 
    for $x\in\R^2\setminus \crit f$.
  \end{enumerate}
 \end{prop}
 To prove the proposition we need
 
\begin{lem}\label{lem:estimates}
 For $i$ large enough we have the estimates
  \begin{equation}\label{it:estimate3} \left\|v_i-(-1)^i\frac{x_i}{\|x_i\|}\right\|\leq \frac{6}{\log i}
  \end{equation}
  and, if $\dist(x_i,y)\leq 2^{1-i}$,
  \begin{equation}\label{it:estimate4}
      \left\|\frac{x_i}{\|x_i\|}-\frac{y}{\|y\|}\right\|\leq 3\dist(x_i,y).
  \end{equation}
\end{lem}
\begin{proof}
 To show \eqref{it:estimate3}, first observe that, in the definition of $x_i$, the jump in the direction tangential to the circle has magnitude $\vartheta_{i}-\vartheta_{i-1}=1/i\log i$, while the jump in the direction normal to the circle has magnitude $\frac1i+\frac{1}{i+1}$. It follows that
 \begin{gather*} 
  \frac{1}{2i\log i}\leq (x_{i+1}-x_i)\cdot \frac{x_i^\perp}{\|x_i\|}\leq \frac{2}{i\log i},\\
  \frac{2}{i+1}\Big(1-\frac{1}{\log i}\Big)\leq \frac{2}{i+1}\sqrt{1-\frac{1}{\log^2i}}\leq (x_{i+1}-x_i)\cdot\frac{x_i}{\|x_i\|}\leq \|x_{i+1}-x_i\|,
 \end{gather*}   
 where $(a,b)^\perp=(-b,a)$ and we have used the Cauchy--Schwarz inequality.
 Since $1/i\leq\varepsilon_i=\|x_{i+1}-x_i\|\leq 2/i$, together with $v_i=-(x_{i+1}-x_i)/\varepsilon_i$ and the estimates above, we also have
 \begin{equation}\label{it:estimate2} 
   \frac{1}{2\log i}\leq\left|v_i\cdot \frac{x_i^\perp}{\|x_i\|}\right|\leq \frac{2}{\log i},\end{equation}   
  and
 \begin{equation}\label{it:estimate1} 
    \frac{i}{i+1}\left(1-\frac2{\log i}\right)\leq  \left|v_i\cdot\frac{x_i}{\|x_i\|}\right|\leq 1.
  \end{equation}

 The estimate \eqref{it:estimate3} follows from \eqref{it:estimate2} and \eqref{it:estimate1}:
 \begin{align*}
  \left\|v_i-(-1)^i\frac{x_i}{\|x_i\|}\right\|
  &=\sqrt{\left(v_i\cdot \frac{x_i}{\|x_i\|}-1\right)^2+\left(v_i\cdot \frac{x^\perp_i}{\|x_i\|}\right)^2}\\
  &\leq \sqrt{\left(\frac{i}{i+1}\left(\frac{2}{\log i}+1\right)-1\right)^2+\left(\frac{2}{\log i}\right)^2}\\
  &\leq \frac{4}{\log i}+\frac2i\leq \frac{6}{\log i}.
 \end{align*}
 
 Estimate \eqref{it:estimate4} can be deduced by letting $w=y-x_i$, so that $\|w\|=\dist(x,y)$ and observing that
 \[1-\tfrac2i\leq\|x_i\|\leq 1+\tfrac2i\quad\textrm{and}\quad \|x_i+w\|=\|y\|\geq 1-\tfrac2i,\]
 which means that, for $i$ large, we have
 \begin{align*}
  \left\|\frac{x_i}{\|x_i\|}-\frac{y}{\|y\|}\right\|
  &=  \left\|\frac{x_i}{\|x_i\|}-\frac{x_i+w}{\|x_i+w\|}\right\|\\
  &=\frac{\big\|\,x_i(\|x_i+w\|-\|x_i\|)+w\|x_i\|\,\big\|}{\|x_i\|\,\|x_i+w\|}\\
  &\leq \frac{2\|x_i\|\,\|w\|}{\|x_i\|\,\|x_i+w\|}\\
  &\leq 2\frac{1+\frac2i}{(1-\frac2i)^2}\|w\|\\
  &\leq 3\|w\|.\qedhere
 \end{align*}
\end{proof}

\begin{proof}[Proof of Proposition \ref{prop:propertiesf}]
 Item \ref{f:regularity} becomes evident once we realize that the sum \eqref{eq:deff} reduces to $f(x)=(1-\psi_i(x))\phi(x)+\psi_i(x)V_i(x)$ for $x$ in $B_{1/2^{i-1}}(x_i)$ and to $f(x)=\phi(x)$ elsewhere, since $\psi_i$, $V_i$ and $\phi$ are $C^\infty$ on $\R^2\setminus (S^1\cup\{0\})$. 
 
 To prove item \ref{f:gradientsequence}, note that, for $x\in B_{1/2^i}(x_i)$, we have $f(x)=V_i(x)$ and $\nabla f(x)=\nabla V_i(x)=v_i$ so that $x_{i+1}-x_i=-\varepsilon_iv_i=-\varepsilon_i\nabla f(x_i)$.
 
 In order to prove item \ref{f:clarke}, let $p\in S^1$. Let us first show that, as $y\in\R^2$ with $\|y\|<1$ tends  $p$, $\nabla f(y)\to -p$. If $y\notin \bigcup_i B_{1/2^{i-1}}(x_i)$ is near $p$, then 
 \[\|\nabla f(y)+p\|=\|\nabla \phi(y)+p\|=\left\|-\frac{y}{\|y\|}+p\right\|,\]
 which clearly tends to 0 as $y\to p$. 
 If $y\in B_{1/2^{i-1}}(x_i)$ (and since $\|y\|<1$ we must have $i$ odd), then
 we have, by a Taylor expansion, $\nabla\phi(x_i)=-x_i/\|x_i\|$, the Cauchy--Schwarz inequality, and \eqref{it:estimate3}, 
 \begin{align*}
  |V_i(y)-\phi(y)|&=\left|(y-x_i)\cdot v_i+\tfrac1i-\phi(x_i)-\nabla\phi(x_i)\cdot(y-x_i)\right|+2\|y-x_i\|^2\\
  &=\left|(y-x_i)\cdot (v_i+\frac{x_i}{\|x_i\|})+\tfrac1i-\tfrac1i\right|+2\|y-x_i\|^2\\
  &\leq 2\|y-x_i\|\left\|v_i+\frac{x_i}{\|x_i\|}\right\|+2\left(\frac{1}{2^{i-1}}\right)^2\\
  &\leq 2\frac1{2^{i-1}}\frac{6}{\log i}=\frac{12}{2^{i-1}\log i}
 \end{align*}
 and, since also $\nabla \phi(y)=-y/\|y\|$, $\nabla V_i(y)=v_i$, $\lip(\nabla\psi_i)=2^i\lip(\nabla\psi)$, $|\psi_i(y)|\leq 1$, the triangle inequality, the estimates from Lemma \ref{lem:estimates}, and $y\in B_{1/2^{i-1}}(x_i)$,
 \begin{align*}
  \left\|\nabla f(y)+\frac{y}{\|y\|}\right\|&=\left\|\nabla[(1-\psi_i(y))\phi(y)+\psi_i(y)V_i(y)]+\frac{y}{\|y\|}\right\|\\
  &=\left\|\nabla \psi_i(y)(V_i(y)-\phi(y))+\psi_i(y)\left(\nabla V_i(y)+\frac{y}{\|y\|}\right)\right\| \\
  &\leq \lip(\nabla \psi_i)|V_i(y)-\phi(y)|+\left\|v_i+\frac{y}{\|y\|}\right\| \\
  &\leq 2^i\lip(\nabla \psi)|V_i(y)-\phi(y)|+\left\|v_i+\frac{x_i}{\|x_i\|}\right\|+\left\|\frac{x_i}{\|x_i\|}-\frac{y}{\|y\|}\right\| \\
  &\leq 2^i\lip(\nabla\psi)\frac{12}{2^{i-1}\log i}+\frac{6}{\log i}+\frac3{2^{i-1}}\\
  &=(12\lip(\nabla\psi)+6)\frac{2}{\log i} +\frac3{2^{i-1}}\to 0\quad\textrm{as $i\to +\infty$.}
 \end{align*}
 It follows from the triangle inequality that
 \[\left\|\nabla f(y)+p\right\|\leq \left\|\nabla f(y)+\frac{y}{\|y\|}\right\|+\left\|p-\frac{y}{\|y\|}\right\|\]
 so that, as $y\to p$ with $\|y\|<1$, we have $\nabla f(y)\to -p$.
 A similar argument yields that, as $y\to p$ with $\|y\|>1$, we have $\nabla f(y)\to p$, which proves item \ref{f:clarke}.

 To prove item \ref{f:pathdiff}, note that, by items \ref{f:regularity} and  \ref{f:clarke}, if a Lipschitz curve $\gamma$ satisfies either $\gamma(t)\in S^1$ and $\gamma'(t)$ tangent to $S^1$ or $\gamma(t)\in\R^2\setminus S^1$, then indeed we have $(f\circ\gamma)'(t)=v\circ\gamma'(t)$ for all $v\in\partial^cf(\gamma(t))$. On the other hand, the set of points $t$ in the domain of $\gamma$ such that $\gamma(t)\in S^1$ but $\gamma'(t)$ is not tangent to $S^1$ is at most countable (these points $t$ can be covered by disjoint open sets) and hence has measure zero; see also the proof of \cite[Theorem 5.3]{Davis2019}. It follows that the chain rule condition for path differentiability is satisfied for almost all $t$. Since this is true for all curves $\gamma$, $f$ is path-differentiable.
 
  Item \ref{f:whitney} is clear in view of items \ref{f:regularity} and \ref{f:clarke}.
 
 If follows from item \ref{f:clarke} that $S^1\subseteq \crit f$.
 Recall $f=\phi$ in a neighborhood of $0$ and $0\in\crit \phi$, so $0\in\crit f$. If $x\notin \bigcup_iB_{1/2^{i-1}}(x_i)$, then $\|\nabla f(x)\|=\|\nabla \phi(x)\|=1$ and $\nabla f(x)$ is the only element of $\partial^cf(x)$, so $x\notin\crit f$. If $x\in B_{1/2^{i-1}}(x_i)$, then, taking $i_0$ large enough, we can ensure that, for $i\geq i_0$, we have, by the triangle inequality and the estimates above,
 \[\|\nabla f(x)\|\geq \left\|\frac{x}{\|x\|}\right\|- \left\|\nabla f(x)-\frac{x}{\|x\|}\right\|%
 >\frac12.
 \]
 This settles items \ref{f:whitney} and \ref{f:KL}.
\end{proof}

 \subsection{Limiting measures}\label{sec:circlemeasures}
 
 Here we recall some of the theory of \cite[Section 4]{ouroscillationcompensation}, and we show that in the example constructed in Section \ref{sec:circleconstruction}, the set of limiting measures is uncountable. We also compute those measures explicitly.
 
 \paragraph{The interpolating curve and its associated closed measures.}

Given a measure $\xi$ on $X$ and a measurable map $g\colon X\to Y$, the \emph{pushfoward $g_*\xi$} is defined to be the measure on $Y$ such that, for $A\subset Y$ measurable, $g_*\xi(A)=\xi(g^{-1}(A))$.

Recall that the \emph{support $\supp\mu$} of a positive Radon measure $\mu$ on $\R^n$ is the set of points $x\in \R^n$ such that $\mu(U)>0$ for every neighborhood $U$ of $x$. It is a closed set.

\begin{defn}\label{def:closedmeasure}
 A compactly-supported, positive, Radon measure on $\R^n\times\R^n$ is \emph{closed} if, for all functions $f\in C^\infty(\R^n)$, 
 \[\int_{\R^n\times\R^n}\nabla f(x)\cdot v\,d\mu(x,v)=0.\]
\end{defn}

Let $\pi\colon\R^n\times\R^n\to\R^n$ be the projection $\pi(x,v)=x$. To a measure $\mu$ in $\R^n\times\R^n$ we can associate its \emph{projected measure $\pi_*\mu$}. We have $\supp\pi_*\mu=\pi(\supp\mu)\subseteq\R^n$. 

Let $\gamma\colon\R_{\geq 0}\to\R^n$ be the curve linearly interpolating the sequence $\{x_i\}_i$ with $\gamma(t_i)=x_i$ for $t_i=\sum_{j=0}^{i-1}\varepsilon_i$ and $\gamma'(t)=v_i$ for $t_i<t<t_{i+1}$.

\medskip

For a bounded set $B\subset\R_{\geq0}$, we define a measure on $\R^n\times\R^n$ by
\[\meas{\gamma}{B}=\frac{1}{|B|}(\gamma,\gamma')_*\lebesgue_{B},\]
where $|B|=\int_B 1\,dt$ is the length of $B$, and $\lebesgue_B$ is the Lebesgue measure on $B$ (so that $\lebesgue_B(A)=|A|$ for $A\subseteq B$ measurable). 
If $\varphi\colon\R^n\times\R^n\to\R$ is measurable, then
\[\int_{\R^n\times\R^n}\varphi\,d\meas{\gamma}{B}=\frac1{|B|}\int_B\varphi(\gamma(t),\gamma'(t))\,dt.\]

\begin{lem}[{\cite[Lemmas 20 and 21]{ouroscillationcompensation}}]\label{lem:longintervals}
 In the weak* topology, the set of limit points of the sequence $\{\meas{\gamma}{[0,N]}\}_N$ is nonempty, and its elements are closed probability measures. Also,
 \[\overline{\bigcup_{\mu\in\acc\{\meas{\gamma}{[0,N]}\}_N}\pi(\supp\mu)}=\essacc\{x_i\}_i.\]
\end{lem}

A measure $\mu$ on $\R^n\times\R^n$ can be fiberwise disintegrated as
\[\mu=\int_{\R^n}\mu_x\,d\pi_*\mu(x),\]
where $\mu_x$ is a probability on $\R^n$ for each $x\in\R^n$.
We define the \emph{centroid field $\bar v_x$} of $\mu$ by
\[\bar v_x=\int_{\R^n}v\,d\mu_x(v).\]
An important intermediate result of \cite{ouroscillationcompensation} is

\begin{thm}[Subgradient-like closed measures are trivial {\cite[Theorem 23]{ouroscillationcompensation}}]\label{thm:vxvanishes}
 Assume that $f\colon\R^n\to\R$ is a path-differentiable function.
 Let $\mu$ be a closed measure on $\R^n\times\R^n$, and assume that every $(x,v)\in\supp\mu$ satisfies $-v\in\partial^cf(x)$. Then the centroid field $\bar v_x$ of $\mu$ vanishes for $\pi_*\mu$-almost every $x$.
\end{thm}

\paragraph{Analysis of the example.} 
 Let $\gamma$ be the interpolating curve of the sequence $\{x_i\}_i$, as defined in Section \ref{sec:circleconstruction}.
 In this example, the set of limit points of the sequence $\{\meas{\gamma}{[0,N]}\}_N$ consists of all measures on $T\R^2$ given by
 \begin{equation}\label{eq:mutheta}\mu^{\theta_0}=\int_{\theta_0-2\pi}^{\theta_0}\frac{\delta_{(r(\theta),r(\theta))}+\delta_{(r(\theta),-r(\theta))}}2\frac{e^{\theta-\theta_0}}{1-e^{-2\pi}}\,d\theta,\quad \theta_0\in\R,
 \end{equation}
 where $r(\theta)=(\cos \theta,\sin \theta)$ and $\delta_{(u,v)}$ denotes the Dirac delta in $\R^2\times\R^2$ concentrated at $(u,v)$. This is the measure that captures the dynamics occurring whenever $x_N$ has angle $\vartheta_N$ close to $\theta_0$. Of course, we have $\mu^{\theta_1}=\mu^{\theta_2}$ if $\theta_2-\theta_1$ is an integer multiple of $2\pi$, as well as $R^\xi_*\mu^{\theta_0}=\mu^{\xi+\theta_0}$ for $R^\xi$ the rotation by angle $\xi$.
 
 Before proving \eqref{eq:mutheta}, we remark that in accordance with Theorem \ref{thm:vxvanishes} we have, for $x\in S^1$,
 \[\mu_x^{\theta_0}=\frac{\delta_{(x,x)}+\delta_{(x,-x)}}2\]
 and
 \[\bar v_x=\int_{\R^2} v\,d\mu_x(v)=x-x=0.\]
 Also the conclusion of Lemma \ref{lem:longintervals} is verified: we have \[\essacc\{x_i\}_i=\pi(\supp\mu^{\theta_0})=S^1,\] 
 and each $\mu^{\theta_0}$ is a closed probability measure.

 Let us see how to arrive at \eqref{eq:mutheta}. 
 From the construction, it is clear that these measures must have the form \begin{equation*}
 \int_{\theta_0-2\pi}^{\theta_0}\frac{\delta_{(r(\theta),r(\theta))}+\delta_{(r(\theta),-r(\theta))}}2\varrho(\theta)\,d\theta
 \end{equation*}
 for some density $\rho$ on $\R$;
 the sum of Dirac deltas in \eqref{eq:mutheta} can be deduced from the fact that the vectors $v_i$ asymptotically approach $y$ and $-y$ as $x_i\to y\in S^1$ (with a subsequence), together with $\gamma'(t)=v_i$ for $t_i<t<t_{i+1}$. 
 
 Let us compute the density $\varrho$.
 Let $I\subset\R$ be an interval of length $0<\alpha=|I|\leq 2\pi$. Considering $I$ as an arc in the circle, we will write \[\beta\in I\!\!\!\!\mod 2\pi\] 
 if $\beta\in\R$ and there is some $k\in\Z$ such that $\beta+2\pi k\in I$.
 Let 
 \[m_0=\min\{i:\vartheta_i\in I\!\!\!\!\mod 2\pi\}.\] 
 Writing $P\approx Q$ if $P/Q\to1$ as $N\to+\infty$, if $m<n$ are two integers such that $\alpha=\vartheta_{n}-\vartheta_m=\sum_{k=m}^{n-1}\frac1{k\log k}$, then 
 \[\alpha\approx\int_m^ndx/x\log x=\log \log n-\log\log m;\]
 thus $n\approx m^{e^\alpha}$. 
 In other words, the intervals $J\subset\N$ such that $\vartheta_i+2\pi k\in I\!\!\mod 2\pi$  if $i\in J$ are approximately
 \[[m_0,m_0^{e^\alpha}],\;[m_0^{e^{2\pi}},m_0^{e^{\alpha+2\pi}}],\;[m_0^{e^{4\pi}},m_0^{e^{\alpha+4\pi}}],\;\dots,\;[m_0^{e^{2k\pi}},m_0^{e^{\alpha+2k\pi}}],\;\dots\]
 Letting $k_N\in\N$ be such that $N=m_0^{e^{\alpha+2\pi k_N}}$, we compute
 \begin{align*}
 \frac{\sum_{\substack{\vartheta_i\in I\\i\leq N}}\varepsilon_i}{\sum_{i=2}^N\varepsilon_i}
 &\approx \frac{\sum_{\substack{\vartheta_i\in I\\i\leq N}}2/i}{\sum_{i=2}^N2/i}\\
 &\approx\frac{1}{\log N}\sum_{k=0}^{k_N}\int_{m_0^{e^{2k\pi}}}^{m_0^{e^{\alpha+2k\pi}}}\frac{dx}x\\
 &=\frac{1}{\log N}\sum_{k=0}^{k_N}(e^\alpha-1)e^{2k\pi}\log m_0\\
 &=\frac{(e^\alpha-1)\log m_0}{\log N}\frac{e^{2\pi( k_N+1)}-1}{e^{2\pi}-1}\\
 &=\frac{(e^\alpha-1)\log m_0}{\log N}\frac{e^{2\pi-\alpha}\log N/\log{m_0}-1}{e^{2\pi}-1}\\
 &\to\frac{1-e^{-\alpha}}{1-e^{-2\pi}}\eqqcolon p(\alpha)
 \end{align*}
 as $N\to+\infty$.
 To compute $\varrho$, we apply that to an interval $I$ of the form $[\theta,\theta_0]=[\theta_0-\alpha,\theta_0]$ and we take the derivative
 \[\varrho(\theta)=\frac{dp(\theta_0-\theta)}{d\theta}=\frac{d}{d\theta}\frac{1-e^{-(\theta_0-\theta)}}{1-e^{-2\pi}}=\frac{e^{\theta-\theta_0}}{1-e^{-2\pi}},\quad \theta\in[\theta_0-2\pi,\theta_0).\]

\section{Example on a fractal set}
\label{sec:sardcounterexample}
In the spirit of Whitney's counterexample \cite{whitney} to the Morse--Sard theorem, we construct a function $f\colon\R^2\to\R$ and a bounded subgradient sequence $\{x_i\}_i$ satisfying:
\begin{enumerate}[label=C\arabic*.,ref=C\arabic*]
 \item\label{itex:pathdiff}\label{itex:first} $f$ is path-differentiable,
 \item\label{itex:nonconstcrit} $f(\crit f)\supset f(\essacc\{x_i\}_i)= f(\acc\{x_i\}_i)=[0,1]$,
 \item\label{itex:essacc} The accumulation set $\acc\{x_i\}_i$ is not contained in $\crit f$, and \[\essacc\{x_i\}_i\neq \acc\{x_i\}_i.\]
 \item\label{itex:fconv} $\{x_i\}_i$ and $\{f(x_i)\}_i$ do not converge.
 \item \label{itex:dim} The Hausdorff dimensions of $\essacc\{x_i\}_i$ and $\acc\{x_i\}_i$ are greater than 1 and satisfy \eqref{eq:fractaldimension}.
 \item \label{itex:slowdown} There are points $x$ and $y$ in $\essacc\{x_i\}_i\setminus\acc\{x_i\}_i$ such that we can take subsequences $\{x_{i_k}\}_k$ and $\{x_{i'_k}\}_k$ converging to $x$ and $y$, respectively, with $i_k< i'_k<i_{k+1}$ for all $k$ and \[\sup_k\sum_{p=i_k}^{i'_k}\varepsilon_p<+\infty.\]
 \item \label{itex:osccomp} There is no oscillation compensation on $\acc\{x_i\}_i\setminus\essacc\{x_i\}_i$. This means, precisely, that there is a continuous function $Q\colon\R^n\to[0,1]$ such that
 \begin{equation}\label{eq:speedaverages}
  \liminf_{N\to+\infty}\left\|\frac{\sum_{i=0}^{N}\varepsilon_iv_i Q(x_i)}{\sum_{i=0}^{N}\varepsilon_i Q(x_i)}\right\|>0.
 \end{equation}
 Crucially, since we strive to show that the dynamics on $\acc\{x_i\}_i\setminus\essacc\{x_i\}_i$ may be very different to the one displayed on $\essacc\{x_i\}_i$, we are not requiring the condition from \cite[Theorem 6(ii)]{ouroscillationcompensation}, namely, the existence of a sequence $\{N_i\}_i$ with 
 \[\liminf_{j\to+\infty}\frac{\sum_{i=1}^{N_j}\varepsilon_i Q(x_i)}{\sum_{i=0}^{N_j}\varepsilon_i}>0,\]
 which would force the focus to be on the dynamics around $\essacc\{x_i\}_i$.
 \item \label{itex:perp}\label{itex:last} The oscillations near $\acc\{x_i\}_i\setminus\essacc\{x_i\}_i$ are not asymptotically perpendicular to $\acc\{x_i\}_i$.
\end{enumerate}

\paragraph{Outline.}
To construct the function $f$, we will first define a fractal curve $\Gamma$ and $f$ on it, aiming to have $\Gamma\subset\crit f$ and $f(\Gamma)=[0,1]$.  We will also define a curve $J$ such that $\Gamma\cup J$ is a closed loop and $J$ only intersects $\crit J$ at its endpoints. We will construct an auxiliary path-differentiable function $h$ coinciding with $f$ on the curve $\Gamma$, and in Lemma \ref{lem:defh} we will prove some properties of $h$. We will next construct a series of loops $T_0,T_1, T_2,\dots$ that will help us define the sequence $\{x_i\}_i$, which we carefully specify so that it is almost a subgradient sequence of $h$. The dynamics of $\{x_i\}_i$ around $\Gamma$ will mimic that of the sequence in the example of Section \ref{sec:circle}, and near $J$ it will instead move relatively fast. To obtain $f$, we modify $h$ slightly in a way that ensures that $\{x_i\}_i$ is a subgradient sequence. In Proposition \ref{prop:propertiesf2} we show that $f$ has certain properties, which we will finally link, in our concluding remarks, to claims \ref{itex:first}--\ref{itex:last} above. 

The reader will find this example easier to follow after having looked at the construction of Section \ref{sec:circleconstruction}. The role of the function $\phi$ in that construction is taken by the function $h$ in the one presented below.

\paragraph{ A fractal curve.} Pick $\frac14<\alpha\leq\frac13$.
We begin by constructing a set $\Gamma\subset\R^2$ recursively as illustrated in Figure \ref{fig:Gamma}. For the first step, we pick four disjoint squares of side $\alpha$ inside the unit square, and we let $\Gamma_1$ be the closed set  consisting of the five disjoint paths joining the left and bottom sides of the unit square with those four squares successively, as in the figure. In each of the following inductive steps, we rescale the set $\Gamma_{i}$ we had for the previous step and we place new copies inside each of the four squares, perhaps rotated by an angle $\pi/2$, so that the paths making up $\Gamma_1$ connect with those of each rescaled copy of $\Gamma_i$. The set $\Gamma_{i+1}$ is then the union of $\Gamma_1$ with the four rescaled and appropriately rotated copies of $\Gamma_i$. This defines an increasing sequence of sets $(\Gamma_i)_{i\in \N}$ and $\Gamma=\overline{\bigcup_i\Gamma_i}$.

\begin{figure}
  \centering
    \includegraphics[width=\textwidth]{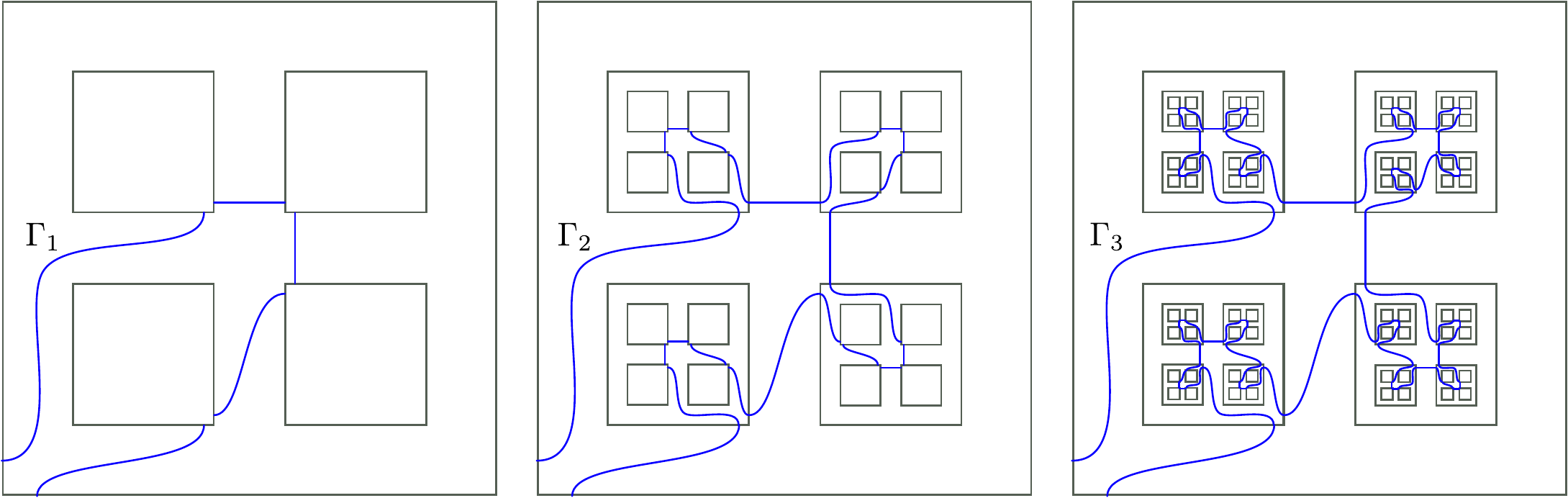}
  \caption{The first three steps $\Gamma_1$, $\Gamma_2$, and $\Gamma_3$ in the construction of the set $\Gamma$.}
  \label{fig:Gamma}
\end{figure}

We proceed to parameterize $\Gamma$ with a continuous curve $\phi\colon[0,1]\to\R^2$. To do this, we will imitate the procedure in the construction of the Cantor staircase. Thus, we first divide $[0,1]$ into nine contiguous intervals of equal length, namely, the nine intervals (we write in base 9) 
\[[0_9,0.1_9),[0.1_9,0.2_9),\dots,[0.7_9,0.8_9),[0.8_9,1_9].\] 
We define the map $\phi$ on each of the of the five odd-numbered intervals
\[[0_9,0.1_9),[0.2_9,0.3_9),[0.4_9,0.5_9),[0.6_9,0.7_9),[0.8_9,1_9]\]
to map the corresponding interval to one of the intervals making up $\Gamma_1$ (in Figure \ref{fig:Gamma}, these are the five blue curves in the left-hand diagram). Then iteratively, at step $i$, we divide each of the remaining intervals into nine equal subintervals, and we map the odd-numbered subintervals into the pieces of $\Gamma_i\setminus\Gamma_{i-1}$.
Thus for example the interval $[0.1_9,0.2_9)$ gets divided into
 \[[0.1_9,0.11_9),[0.11_9,0.12_9),\dots,[0.17_9,0.18_9),[0.18_9,0.2_9),\]
and the images of the intervals $[0.1_9,0.11_9)$ and $[0.18_9,0.2_9)$ will touch the images of the intervals $[0.0_9,0.1_9)$ and $[0.2_9,0.3_9)$, but the intervals
\[[0.12_9,0.13_9),[0.14_9,0.15_9),[0.16_9,0.17_9)\]
will not touch the image of the curve defined in the previous step; refer to the middle diagram in Figure \ref{fig:Gamma}.
The map $\phi$ is the unique continuous extension of the thus-defined function. 

The resulting curve $\phi$ has infinite arc length. Indeed at each construction step of the $\Gamma_i$, the paths in $\Gamma_i\setminus\Gamma_{i-1}$ are contained in $4^{i-1}$ squares, each of them contributing in an increase of at least $2\alpha^i$ in the total length. This results in a global increase of at least $(4\alpha)^i/2>1/2$ in the $i$-th step. 

Let $p\colon[0,1]\to\N\cup\{+\infty\}$ be the function that assigns to a number $t$ the first appearance of an even digit after the decimal point in its base 9 expansion, so that for example $p(0_9)=1$ and $p(0.757823_9)=4$. Thus if $t\in[0,1]$ and $p(t)<+\infty$, then $\phi(t)\in\Gamma_{p(t)}\setminus\Gamma_{p(t)-1}$, and if $p(t)=+\infty$ then $\phi(t)$ is a point in the Cantor set $\Gamma\setminus\bigcup_i\Gamma_i$ at the intersection of all the squares used in the construction.

\begin{lem}\label{lem:dimgamma}
 The Hausdorff dimension of $\Gamma$ is $\log_\alpha\frac14$.
\end{lem}
\begin{proof}
 The definition of Hausdorff dimension was recalled in \ref{q:10} in Section \ref{sec:intro}.
 
 Let $r>0$. As explained above, the length of $\Gamma_i\setminus\Gamma_{i-1}$ is at least $(4\alpha)^i/2$. Thus, a lower bound on the number of balls of radius $r$ necessary to cover $\Gamma_i\setminus\Gamma_{i-1}$ is $(4\alpha)^i/2r-1$ balls, for $i$ such that $\alpha^i>r$, i.e., $i<\log_\alpha r$. We have, for $d>0$,
 \begin{align*}
  \mathcal H^d(\Gamma)&\geq\mathcal H^d(\textstyle\bigcup_i\Gamma_i)\\
  & \geq\liminf_{r\to 0}\sum_{i=1}^{\log_\alpha r-1}r^d\left(\frac{(4\alpha)^i}{2r}-1\right)\\
  &=\liminf_{r\to0}\frac12r^{d-1}\frac{(4\alpha)^{\log_\alpha r}-1}{4\alpha-1}-(\log_\alpha r-1)r^d\\
  &=\liminf_{r\to0}\frac12r^{d-1}\frac{(4\alpha)^{\log_\alpha r}}{4\alpha-1}\\
  &=\liminf_{r\to0}\frac{1/2}{4\alpha-1}\exp(\left(d-1+\log_\alpha(4\alpha)\right)\log r)\\
  &=\liminf_{r\to0}\frac{1/2}{4\alpha-1}\exp(\left(d-\log_\alpha\tfrac14\right)\log r).
 \end{align*}
 Hence in order to have $\mathcal H^d(\Gamma)=0$ it is necessary that $d>\log_\alpha\frac14$ because this $\liminf$ must vanish and $\log r\to-\infty$. This translates to $\dim \Gamma\leq\log_\alpha\frac14$.

 Let us prove the opposite inequality. 
 For $r>0$ we cover  $\Gamma_i\setminus\Gamma_{i-1}$ with $A(4\alpha)^i/r$ balls of radius $r$ for $i$ such that $\alpha^i\geq r$; here $A>0$ is taken so that $4A\alpha^i$ is an upper bound for the contribution of the paths in each of the $4^{i-1}$ squares added.
 Since $\Gamma\setminus\bigcup_i\Gamma_i$ is also the intersection of the squares in the construction above, we know that it can be covered by $4^k$ balls of radius $2\alpha^k$, and these balls will cover the remaining part of $\bigcup_i\Gamma_i$. Hence we have, with $r=2\alpha^k$ and its consequence $\log_\alpha r=k+\log_\alpha2$,
 \begin{align*}
  \mathcal H^d(\Gamma)
  &\leq\mathcal H^d(\Gamma\setminus\textstyle\bigcup_i\Gamma_i)+\mathcal H^d(\textstyle\bigcup_i\Gamma_i)\\
  &\leq \displaystyle\liminf_{k\to+\infty} 4^kr^d+\liminf_{k\to+\infty} \sum_{i=1}^{\log_\alpha r}r^d\frac{A(4\alpha)^i}{r}\\
  &= \displaystyle\liminf_{k\to+\infty} 4^k\big(2\alpha^k\big)^d+\liminf_{k\to+\infty} \sum_{i=1}^{k+\log_\alpha 2}(2\alpha^k)^d\frac{A(4\alpha)^i}{2\alpha^k}\\
  &\leq\liminf_{k\to+\infty}e^{k(\log4+d\log\alpha)+d\log 2}+\liminf_{k\to+\infty} A(2\alpha^k)^{d-1}\frac{(4\alpha)^{k+\log_\alpha 2+1}-1}{4\alpha-1},
 \end{align*}
 which vanishes unless $\log4+d\log\alpha>0$, that is, unless $d<\log_\alpha\frac14$. This gives $\dim \Gamma\geq\log_\alpha\frac14$.
\end{proof}

\paragraph{Defining $f$ on $\Gamma\cup J$.} We define $f$ on $\Gamma$ imitating the construction of the Cantor staircase as follows. For a point $q\in\Gamma_i$, we let $t=\phi^{-1}(q)$, and we express $t$ in base 9, so that the first $k=p(t)-1$ numbers $a_1$, $a_2$, \dots $a_{k}$ in the base 9 expansion $x=(0.a_1a_2a_3\dots)_9$ are odd. We then let, for $1\leq i\leq k$, $b_i=(a_i-1)/2$, and $f(q)=(0.b_1b_2\dots b_k)_4$ in base 4. The values so-assigned for $f$ are illustrated in Figure \ref{fig:fvalues}. The reader will convince herself that with this definition, $f$ is constant on each path-connected component of $\bigcup_i\Gamma_i$ and can be uniquely extended to a continuous function on all of $\Gamma$.
\begin{figure}
  \centering
    \includegraphics[width=\textwidth]{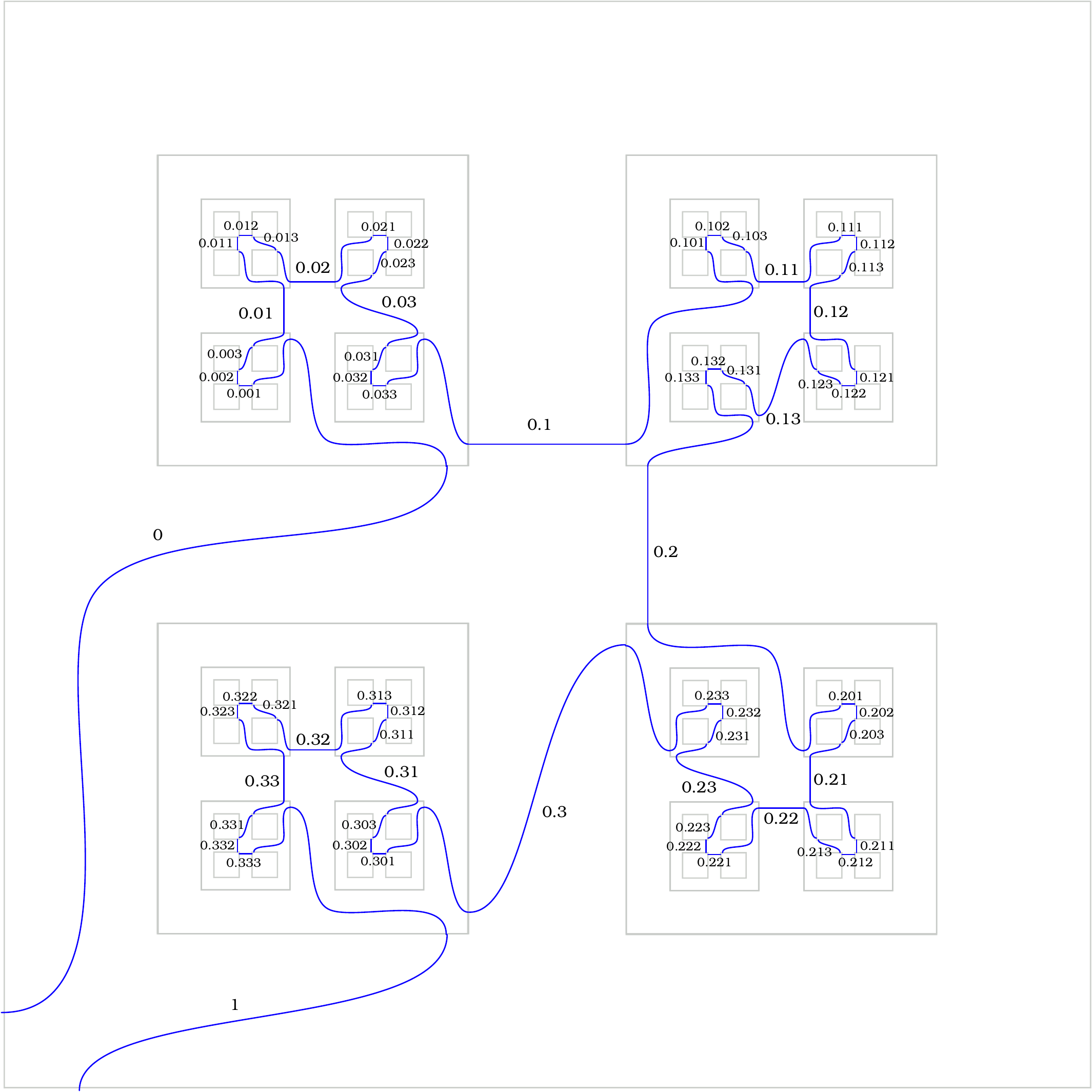}
  \caption{Values of $f$ on the set $\Gamma_3$. All numbers are in base 4.}
  \label{fig:fvalues}
\end{figure}

We remark that the function $f\circ\phi\colon[0,1]\to[0,1]$, just like the Cantor staircase, is continuous but not absolutely continuous; indeed, since it is constant on the intervals where $p$ is constant, its derivative $(f\circ\phi)'$ vanishes almost everywhere on $[0,1]$, yet $f\circ\phi$ is not constant, contradicting the fundamental theorem of calculus, which is valid for absolutely continuous functions. %

Let $J\subset \R^2\setminus((0,1)\times(0,1))$ be a smooth, non-self-intersecting curve joining the two intersections of $\Gamma$ with the boundary of the unit square. We define $f$ on $J$ to smoothly and strictly monotonously take the values between $0$ and $1$, keeping $f$ continuous.

\paragraph{Lipschitz continuity of $f$ on $\Gamma\cup J$.}
Let $j\colon(0,1)\times(0,1)\to\N\cup\{+\infty\}$ be, for each pair of points $s$ and $t$ in $(0,1)$, the position of the first digit of the base-9 expansion $s$ and $t$ that differs; thus for example $j(0.112_9,0.1223_9)=2$.
Since we have $|s-t|>9^{-j(s,t)-1}$, as $|s-t|\to0$ we necessarily have $j(s,t)\to+\infty$. 

Note also that if $j(s,t)>0$, then $\phi(s)$ and $\phi(t)$ must be contained in the same square of side $\alpha^{j(s,t)-1}$.
Thus, for some $A,B>0$,
$|\phi(s)-\phi(t)|\geq A\alpha^{j(s,t)}$ and $|f\circ\phi(s)-f\circ\phi(t)|\leq B 4^{-j(s,t)}$. 

Thus if $x,y\in\Gamma$, letting $s$ and $t$ be such that $\phi(s)=x$ and $\phi(t)=y$, we have $\|x-y\|=|\phi(s)-\phi(t)|\geq A\alpha^{j(s,t)}$ or, equivalently, $-j(s,t)\leq-\log_\alpha \frac{\|x-y\|}A$. Also,
\begin{align*}
 |f(x)-f(y)|
 &=|f\circ\phi(s)-f\circ\phi(t)|\\
 &\leq B4^{-j(s,t)}\\
 &\leq B4^{-\log_\alpha\frac{\|x-y\|}A}\\
 &=\frac BA\|x-y\|^{\log_\alpha\frac14} \\
 &\leq \frac BA\|x-y\|.
\end{align*}
because $\alpha>1/4$ so $\log_\alpha1/4>1$. This, together with the smoothness of $f$ on $J$ implies that $f$ on $\Gamma\cup J$ is Lipschitz. Let $\lip(f|_{\Gamma\cup J})$ be the Lipschitz constant of $f$ on $\Gamma\cup J$.

\paragraph{The auxiliary function $h$.}
 Let $C$ be a connected component of $J\cup\Gamma_i$ for some $i$, without its endpoints. As such, $C$ is a smooth, non-self-intersecting curve, diffeomorphic to an open interval. As is well known (see for example \cite[p. 109]{lang}) there exists a tubular neighborhood $W_C$ around $C$, by which we mean specifically:
 \begin{itemize}
 \item there is an open set $W_C\subset\R^2$ that contains $C$,
 \item there is an open set $U\subset\R^2$ of the form $(a,b)\times (-c,c)$ for some $a,b,c\in\R$, $c>0$, and
 \item there is a smooth, bijective function $\varphi_C\colon \overline U\to \overline{W_C}$ such that
  \begin{itemize} 
  \item the map $x\mapsto\varphi_C(x,0)$ is a parameterization of $C$ by arclength, 
  \item the map $y\mapsto\varphi_C(x,y)$ is a parameterization, by arclength, of the segment perpendicular to $C$ and passing through $x$.
  \end{itemize}
 \end{itemize}   
 We will refer to $\varphi_C$ as the \emph{chart} of $W_C$, and to the number $c>0$ as the \emph{thickness} of $W_C$.
 
 The statement of existence of the tubular neighborhoods $W_C$ is obvious if we choose all $\Gamma_i$ and $J$ to be composed of straight line segments and circle arcs, so readers unfamiliar with the general case may assume that this is the case.
 
\begin{lem}\label{lem:defh}
 There is a function $h\colon\R^2\to\R$ such that 
 \begin{enumerate}[label=\roman*.,ref=(\roman*)]
   \item \label{ith2:pathdiff} $h$ is locally Lipschitz and path-differentiable.
   
   \item \label{ith2:extendsf} $h$ coincides with $f$ on $\Gamma\cup J$.
   
   \item \label{ith2:C1} $h$ is $C^1$ on $\R^2\setminus (\Gamma\cup J)$.
   
   \item \label{ith2:piecewisesmooth} On a tubular neighborhood $W_C$ of each connected component $C$ of $J\cup \Gamma_i$, $h$ is defined by 
   \begin{equation}\label{eq:defh}
    h(\varphi_C(x,y))=f(\varphi_C(x,0))+2L|y|,
   \end{equation}
   where $\varphi_C$ is the chart of $W_C$. Hence
   $h$ is piecewise $C^\infty$ in $W_C$, with the singular locus of $h$  within $W_C$ coinciding exactly with $C$.
   
   \item \label{ith2:clarke} Let $L=\lip(f|_{\Gamma\cup J})$. If $p\in \Gamma_i$ for some $i$, and if $\mathbf n$ is a unit vector normal to $\Gamma_i$ at $p$, then 
   \[\partial^ch(p)=\{\lambda \mathbf n:-2L\leq \lambda\leq 2L\}, \quad p\in \textstyle\bigcup_i \Gamma_i.\]
   More precisely, the gradients of $h$ on each side of $\Gamma_i$ at $p$ are asymptotically equal to $2L\mathbf n$ and $-2L\mathbf n$, respectively, pointing away from $\Gamma_i$.
   
   Similarly, if now $p\in J$, $\mathbf n$ is a unit vector normal to $J$ at $p$, and $\mathbf t$ is the unit vector tangent to $J$ at $p$ that points in the clockwise direction (for the loop $\Gamma\cup J$) and if $a>0$ is the magnitude of the derivative of $f|_J$ at $p$, then
   \[\partial^ch(p)=\{-a\mathbf t+\lambda \mathbf n:-2L\leq \lambda\leq 2L\},\quad p\in J.\]
   More precisely, the gradients of $h$ on each side of $J$ at $p$ are asymptotically equal to $-a\mathbf t+2L\mathbf n$ and $-a\mathbf t-2L\mathbf n$, respectively, pointing away from $J$.
   
   \item \label{ith2:hessian} The norm of the Hessian of $h$ is bounded on each connected component of $W_C\setminus C$, for $C$ and $W_C$ as in item \ref{ith2:piecewisesmooth}.
   
   \item \label{ith2:crit} $\Gamma\subseteq\crit h$.
 \end{enumerate}
\end{lem}

This lemma will be proved in Appendix \ref{sec:proofh}.

\paragraph{A skeleton curve for the sequence.} We shall now define a sequence of smooth loops $T_1,T_2,\dots$  that will guide the trajectory of the sequence $\{x_i\}_i$. Figure~\ref{fig:trajectories} illustrates the shapes of the first elements of the sequence of closed curves that we now proceed to construct. 

The first one, $T_0$, will simply be a small loop around the origin, containing $J=T_0\setminus((0,1)\times(0,1))$ and closing it up with a circular arc contained in $[0,1]\times[0,1]$.

For $i>0$, the path $T_i$ will be equal to $\Gamma_i\cup J$ together with some small circular arcs glued to close up the loose ends in such a way that we obtain a smooth loop that does not touch the $4^{i+1}$ smaller squares of side $\alpha^{i+1}$ involved in the construction of $\Gamma_{i+1}$.

\begin{figure}
  \centering
    \includegraphics[width=\textwidth]{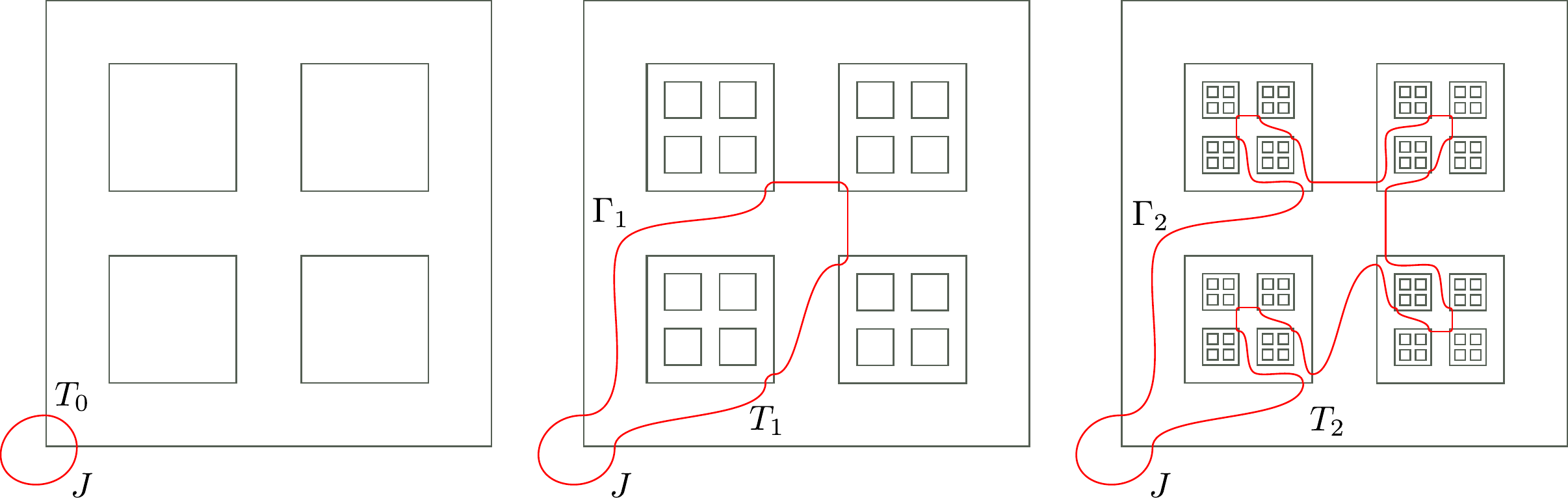}
  \caption{The first three loops, $T_0$, $T_1$, $T_2$, used to define the subgradient sequence $\{x_i\}_i$. The loop $T_i$ contains $J$ and, for $i>0$, it also contains $\Gamma_i$. }
  \label{fig:trajectories}
\end{figure}

\paragraph{Specification of the sequence $\{x_i\}_i$.}

Unlike what we did for the example described in Section \ref{sec:circle}, we will not try here to define $\{x_i\}_i$ explicitly; instead, we will take the lesson from that example as to what this sequence should look like. We pick $\{x_i\}_i$ to be a sequence of distinct points with $\|x_{i+1}-x_i\|\to0$ successively bouncing around each path $T_1, T_2,\dots$. Thus, the sequence will start near $J$, it will go around $T_0$ a few times, and while it is at $J$, it will start going around $T_1$, which it will do a few times, and then $T_2$, and so on.

 Let $L=\lip(f|_{\Gamma\cup J})$ and let $I_0, I_1,\dots\subset\N$ be the intervals during which $x_i$ will be going around each of the paths $T_j$, respectively. We will choose an initial value $i_0>0$ such that the sequence $\{x_i\}_{i=i_0}^\infty$ will satisfy:
\begin{enumerate}[label=S\arabic*.,ref=S\arabic*] 
 \item\label{S:first}\label{S:notselfaccum} Not self-accumulating. We require the sequence $\{x_i\}_i$ to be such that, for each $i$, there is some $r>0$ such that $B_r(x_i)\cap\{x_j\}_{j\neq i}=\emptyset$.
 
 \item\label{S:thickness} If $i\in I_j$ then 
 \[\frac 2{iL}\leq \operatorname{thickness}(W_C)\]
 for all connected components $C$ of $\Gamma_j\cup J$.
 
 \item\label{S:bouncing} Bouncing. If $j>0$ and $i,i+1\in I_j$, the points $x_i$ and $x_{i+1}$ are on opposite sides of $T_j$.
 
 \item\label{S:closetoTi} Distance to $T_j$. For $i\in I_j$, we require the points $x_i$ to remain at a distance  
 \[\left|\dist(x_i,T_j)-\frac{1}{iL}\right|\leq \frac1{i^2}.\]
 
 \item\label{S:aroundsmooth} Around $\Gamma_j\cup J$. Recall from Lemma \ref{lem:defh} that $h$ is piecewise smooth near $\Gamma_j\cup J$. If $j>0$, $i\in I_j$ and the closest point of $T_j$ to $x_i$ is $y\in \Gamma_j\cup J$, and if $\mathbf t$ is the unit vector tangent to $\Gamma_j\cup J$ pointing in the clockwise direction, then we require $h$ to be differentiable at $x_i$ and
 \[\left\|[x_{i+1}-x_i+\frac1{iL}\nabla h(x_i)]\cdot \mathbf t-\frac{1}{i\log i}\right\|\leq\frac1{i^2}.\]
 
 \item\label{S:aroundcloseoff} Around the circle arcs $T_j\setminus(\Gamma_j\cup J)$.  
 If $j>0$, $i\in I_j$, and the point $y$ of $T_j$ closest to $x_i$ is in $T_j\setminus(\Gamma_j\cup J)$,  and if $\mathbf t$ is a unit vector tangent to $T_j$ at $y$ pointing in the clockwise direction, we require
 \[\frac{3}{iL}\leq (x_{i+1}-x_i)\cdot \mathbf t\leq \frac{4}{iL}\]
 
 \item\label{S:last}\label{S:nobigjumps} Small jumps. 
 For all $i$ in the situation of \ref{S:aroundsmooth},
 \[\|x_{i+1}-x_i+\frac1{iL}\nabla h(x_i)\|\leq \frac{3}{iL}.\]
 For all $i$ in the situation of \ref{S:aroundcloseoff},
 \[\|x_{i+1}-x_i\|\leq \frac{6}{iL}.\]
\end{enumerate}

Let us explain how such a sequence can be constructed. First, we choose $i_0>0$ large enough that if $C_1$ is the connected component of $J\cup \Gamma_1$ containing $J$, then $2/i_0L\leq \operatorname{thickness}(W_{C_1})$. We then choose $x_{i_0}$ in $W_{C_1}$ such that the point of $C_1$ closest to $x_{i_0}$ is in $J$, and such that \ref{S:closetoTi} is satisfied with $j=0$. 

By induction, assuming that for some $i\geq i_0$ we have chosen $x_i$ satisfying \ref{S:first}-\ref{S:last}, we let $x_{i+1}$ be a point in the component on the opposite side of $T_j$ (thus complying with \ref{S:bouncing}) of the nonempty set $X_i$ determined by \ref{S:closetoTi} and \ref{S:nobigjumps} together with either \ref{S:aroundsmooth} or \ref{S:aroundcloseoff}, depending on the location of $x_i$. The set $X_i$ is indeed nonempty because the inequality in \ref{S:closetoTi} determines two stripes going parallel to $T_j$, while \ref{S:aroundsmooth} and \ref{S:aroundcloseoff} determine stripes perpendicular to $T_j$. So they intersect (with at least one connected component of the intersection on each side of $T_j$) as long as the step size is small enough with respect to the curvature of $T_j$; this can be ensured in the case of $j=0$ by increasing $i_0$, and in the case of $j>0$ by increasing the amount of times the sequence goes around $T_{j-1}$ before moving on to $T_j$. Although the intersection of the condition in \ref{S:closetoTi} and those of either \ref{S:aroundsmooth} or \ref{S:aroundcloseoff} may also include points located far from $x_i$, \ref{S:nobigjumps} forces the  choose a connected component that is directly ahead along $T_j$, and it is impossible that the sequence will jump very far. Thus \ref{S:bouncing}--\ref{S:nobigjumps} can be complied with.

To see that \ref{S:notselfaccum} can be complied with as well, note that, by \ref{S:closetoTi}, together with \ref{S:aroundsmooth} and \ref{S:aroundcloseoff}, a ball of radius $r=1/2i^2$ works automatically once the other conditions have been satisfied.
To ensure \ref{S:thickness} is true, we let the sequence go around each $T_j$ a few times until $i$ grows enough that the inequality in \ref{S:thickness} becomes true.

We remark that the precise form of \ref{S:aroundcloseoff} will not be used explicitly, and its only purpose is to keep the sequence moving around the circular arcs $T_j\setminus (\Gamma_j\cup J)$ at a moderate rate.

\paragraph{Construction of $f$.} 
Choose real numbers $\{r_i\}_i\subset\R$ such that $0<r_i<1/i^2$ and such that the disk $B_{3r_i}(x_i)$ of radius $3r_i$ centered at $x_i$ does not intersect $\Gamma$ and all the disks $B_{3r_i}(x_i)$ are disjoint. This is possible because of our specification \ref{S:notselfaccum}.

Let $\psi\colon\R^2\to[0,1]$ be a $C^\infty$ function with radial symmetry, $\psi(x)=\psi(y)$ for $\|x\|=\|y\|$, such that $\psi(x)=1$ for $x\in B_1(0)$, $\psi(x)=0$ for $\|x\|\geq 2$, and decreases monotonically on rays emanating from the origin. Let
 \[\psi_i(x)=\psi\Big(\frac{x-x_i}{r_i}\Big),\]
 so that $\psi_i$ equals 1 on $B_{r_i}(x_i)$ and vanishes outside $B_{2r_i}(x_i)$. Denote by $\lip(\psi)>0$ the Lipschitz constant of $\psi$, and by $\lip(\nabla\psi)>0$ the Lipschitz constant of its gradient. Note that the supports of the functions $\psi_i$ are pairwise disjoint and $\lip(\nabla\psi_i)=\frac{1}{r_i}\lip(\nabla\psi)$.

 Let %
 \[v_i=iL(x_i-x_{i+1}),\]%
 and define, for $h$ as in Lemma \ref{lem:defh},
 \[V_i(x)=(x-x_i)\cdot v_i+h(x_i).\]

 \begin{prop}\label{prop:propertiesf2} Let 
 \begin{equation}\label{eq:deff2}
  f(x)=\left(1-\sum_{i=0}^\infty\psi_i(x)\right)h(x)+\sum_{i=0}^\infty\psi_i(x)V_i(x).
 \end{equation}
 Then we have
  \begin{enumerate}[label=\roman*.,ref=(\roman*)]
   \item\label{f2:regularity} $f$ is piecewise $C^\infty$ in a tubular neighborhood $W_C$ of each connected component $C$ of $\Gamma_i\cup J$, $i>0$. 
   \item\label{f2:gradientsequence} $\{x_i\}_i$ is a subgradient sequence for $f$ with stepsizes
    \[\varepsilon_i=\frac1{iL}.\]
    In particular, $\sum_{i}\varepsilon_i=+\infty$ and $\sum_{i}\varepsilon_i^2<+\infty$.
   \item\label{f2:acc} $\acc\{x_i\}_i=\Gamma\cup J$.
   \item\label{f2:clarke} Let $p$ be a point in $\Gamma_i\cup J$ for some $i>0$. Then 
    \[\partial^c f(p)=\partial^ch(p).\]
   \item\label{f2:crit} The critical set of $f$ contains $\Gamma$, but $J\cap\crit f$ consists only of the two endpoints of $J$.
   \item\label{f2:pathdiff} $f$ is locally Lipschitz and path-differentiable.
  \end{enumerate}
 \end{prop}
 \begin{proof}
  By item \ref{ith2:piecewisesmooth} in Lemma \ref{lem:defh} we know that $h$ is piecewise $C^\infty$ in a tubular neighborhood $W_C$ of each connected component $C$ of $\Gamma_i\cup J$. Item \ref{f2:regularity} then follows from the facts that $V_i$ is $C^\infty$ and that the supports of the functions $\psi_i$ are piecewise disjoint, and the form of \eqref{eq:deff2}.
  
  From \eqref{eq:deff2} and the fact that $\nabla \psi_j(x_i)=0=1-\sum_k\psi_k(x_i)$ for all $i$ and all $j$, we have
  \[\nabla f(x_i)=\nabla V_i(x_i)=v_i.\]
  Thus $\partial^cf(x_i)=\{v_i\}$ and
  \[x_i-\varepsilon_iv_i=x_i-\frac1{iL}iL(x_{i+1}-x_i)=x_{i+1},\]
  which is the statement of item \ref{f2:gradientsequence}.
  
  Note that \ref{S:aroundsmooth} and \ref{S:aroundcloseoff} force the sequence to always advance around each $T_j$ and finish the loop.
  Item \ref{f2:acc} is then clear from the construction of $\Gamma$ and the loops $T_j\supset J$, together with the specification \ref{S:closetoTi} that forces the sequence $\{x_i\}_i$ to get ever closer to $\Gamma\cup J$.

  Let us prove item \ref{f2:clarke}. Fix $j>0$ and $p\in \Gamma_j\cup J$, and denote by $C$ the connected component of $\Gamma_j\cup J$ that contains $p$. Consider a point $y$ near $p$. In particular, we may assume that $y$ is not in the situation described in \ref{S:aroundcloseoff}. If $y\notin\bigcup_iB_{2r_i}(x_i)$, then $f=h$ on a neighborhood of $y$ and we have nothing to prove. Otherwise, we have $y\in B_{2r_i}(x_i)$ for some $i\geq 0$, and by \ref{S:thickness} we may assume $B_{2r_i}(x_i)$ is contained in the neighborhood $W_C$ of item \ref{ith2:piecewisesmooth} in Lemma \ref{lem:defh}. Item \ref{ith2:hessian} in Lemma \ref{lem:defh} means that the derivative of $\nabla h$ (the Hessian of $h$) is bounded on $W_C$, which means in particular that $\nabla h$ is Lipschitz in $W_C$; in other words, there is some $K>0$, depending only on $C$ such that, for all $z\in B_{2r_i}(x_i)$,
  \[\|\nabla h(z)-\nabla h(x_i)\|\leq K\|z-x_i\|.\]
  Note that it follows from \ref{S:bouncing}, \ref{S:closetoTi}, and \ref{S:aroundsmooth} and item \ref{ith2:clarke} of Lemma \ref{lem:defh} that, if $i$ is large enough,
  \begin{equation}\label{eq:oldS4}
   \varepsilon_i\left\|v_i-\nabla h(x_i)-\frac L{\log i}\mathbf t\right\|\leq \frac2{i^2}.
  \end{equation}
  By \eqref{eq:deff2}, the Lipschitzity of $\nabla \psi$, the fact that $0\leq \psi_i\leq 1$, a Taylor expansion with $w$ a point in the segment joining $x_i$ and $y$, the definition of $K$, the Cauchy-Schwarz and triangle inequalities, and \eqref{eq:oldS4}, %
  \begin{align*}
  \|&\nabla f(y)-\nabla h(y)\|=\left\|\nabla[(1-\psi_i(y))h(y)+\psi_i(y)V_i(y)]-\nabla h(y)\right\|\\
  &=\left\|\nabla \psi_i(y)(V_i(y)-h(y))+\psi_i(y)\left(\nabla V_i(y)-\nabla h(y)\right)\right\| \\
  &\leq \lip(\nabla \psi_i)|V_i(y)-h(y)|+\left\|v_i-\nabla h(y)\right\| \\ 
  &\leq \frac1{r_i}\lip(\nabla\psi)(\|h(x_i)+v_i\cdot(y-x_i)\\
  &\qquad\qquad\qquad-h(x_i)-\nabla h(w)\cdot(y-x_i)\|)+\|v_i-\nabla h(y)\|\\
  &\leq \frac1{r_i}\lip(\nabla\psi)\|v_i-\nabla h(w)\|\,\|y-x_i\|+\|v_i-\nabla h(y)\|\\
  &\leq \frac1{r_i}\lip(\nabla\psi)(2r_i)(\|v_i-\nabla h(x_i)\|+\|\nabla h(x_i)-\nabla h(w)\|)\\
  &\qquad\qquad\qquad+\|v_i-\nabla h(x_i)\|+\|\nabla h(x_i)-\nabla h(y)\|\\
  &\leq 2\lip(\nabla\psi)(\|v_i-\nabla h(x_i)\|+2Kr_i)+\|v_i-\nabla h(x_i)\|+2Kr_i\\
  &\leq %
  (2\lip(\nabla\psi)+1)\left(\frac{4L}{\log i} +2Kr_i\right)\to 0,
 \end{align*}
 as $y\to p$ because, in that case, $i\to +\infty$. So item \ref{f2:clarke} follows.
 
 Item \ref{f2:crit} follows from item \ref{f2:clarke} together with the same being true for $h$; see items \ref{ith2:clarke} and \ref{ith2:crit} in Lemma \ref{lem:defh}.
 
 Item \ref{f2:pathdiff} follows from item \ref{ith2:pathdiff} in Lemma \ref{lem:defh}, the form of \eqref{eq:deff2} on $\R^2\setminus(\Gamma\cup J)$, which ensures that the path differentiability of $h$ is inherited by $f$ on that region, and from item \ref{f2:clarke} above, which ensures the modification \eqref{eq:deff2} of $h$ does not change the path differentiability property on $\Gamma\cup J$.
 \end{proof}
 
 \begin{lem}\label{lem:essacc}
  $\essacc\{x_i\}_i=\Gamma$.
 \end{lem}
 \begin{proof}
  We will first show that $\bigcup_i \Gamma_i\subseteq\essacc\{x_i\}_i$, and from the fact that $\essacc\{x_i\}_i$ is closed it will follow that $\Gamma$ is contained in it. We use the notation $P\approx Q$ to mean that $P/Q\to1$.
  
  Let $j>0$, $p\in\Gamma_j$ that is not an endpoint of the connected component $C$ of $\Gamma_j$ containing $p$, and $\{N_i\}_i\subset\N$ be a subsequence such that $\lim_ix_{N_i}=p$. Let $\alpha>0$ be smaller than the distance between $p$ and the closest of the two endpoints of $C$. Let also $\{M_i\}_i\subset \N$ be a subsequence such that $q=\lim_{i}x_{M_i}$ is a point on $C$ at arclength $\alpha$ from $p$, $M_i<N_i$ for all $i$, and $\dist(x_k,C)<2/{kL}$ for all $M_i\leq k\leq N_i$. In view of \ref{S:closetoTi}, for each $i$ the sequence $x_{M_i},x_{M_i+1},\dots,x_{N_i}$ is bouncing around the segment of $C$ of length $\alpha$ that starts at $q$ and ends at $p$.  By item \ref{ith2:clarke} of Lemma \ref{lem:defh}, we know that $\partial^ch$ on the points of $C$ contains only vectors that are normal to $C$, so \ref{S:aroundsmooth} implies that 
  \[\tfrac12\alpha\leq \sum_{k=M_i}^{N_i}\frac{1}{i\log i}\approx \log\log N_i-\log\log M_i.\]
  This means that 
  \[\log M_i\leq \exp(\log\log N_i-\tfrac12\alpha)=e^{-\alpha/2}\log N_i.\]
  Hence also
  \[\sum_{k=M_i}^{N_i}\varepsilon_k=\sum_{k=M_i}^{N_i}\frac1{kL}\approx\frac1L(\log N_i-\log M_i)\geq \frac1L(1-e^{-\alpha/2})\log N_i.\]
  Similarly,
  \[\sum_{k=1}^{N_i}\varepsilon_k=\sum_{k=1}^{N_i}\frac1{kL}\approx\frac1L\log N_i.\]
  Thus the $\limsup$ in the definition \eqref{eq:essaccdef} of $\essacc\{x_i\}_i$ is at least $1-e^{-\alpha/2}>0$. This proves that $p\in \essacc\{x_i\}_i$, and thus also that $\Gamma\subseteq\essacc\{x_i\}_i$.
  
  In view of item \ref{f2:acc} of Proposition \ref{prop:propertiesf2} and the fact that $\essacc\{x_i\}_i\subseteq\acc\{x_i\}_i=\Gamma\cup J$, we now need to show that if $p'\in J\setminus \Gamma$ then $p'\notin\essacc\{x_i\}_i$.
  For such $p'$ we pick an open ball $U$ containing $p'$ such that $\overline U\cap \Gamma=\emptyset$ and 
  \[\kappa_U\coloneqq\inf_{x\in U}\dist(0,\partial^cf(x))>0,\]
  as is possible because of item \ref{f2:clarke} of Proposition \ref{prop:propertiesf2}, together with the fact that $f$ is strictly monotonous on $J$. Let $a>0$ be the arclength of $J\cap U$. Then from \ref{S:aroundsmooth} it follows that if $i_1<i_2$ are such that for all $i_1\leq k\leq i_2$ we have $x_k\in U$, while $x_{i_1-1},x_{i_2+1}\notin U$, then 
  \[2a\geq \sum_{k=i_1}^{i_2}\varepsilon_k\|v_k\|\geq \sum_{k=i_1}^{i_2}\varepsilon_k\kappa_U\]%
  For $\ell>0$, let $p_\ell$ denote the number of times the sequence goes around $T_\ell$.
  If $N>0$ is in $I_j$, so that the sequence is bouncing around $T_j$, then 
  \[\sum_{\substack{x_i\in U\\i\leq N}}\varepsilon_k\leq \frac{2a}{\kappa_U}\sum_{\ell=0}^jp_\ell.\]
  On the other hand, to estimate $j$ as a function of $N$ we compute a lower bound of the length of the path traversed by $x_1,\dots, x_N$,
  \begin{align*}
   \sum_{i=1}^{j-1} p_i\arclength \Gamma_i&\geq
   \sum_{i=1}^{j-1}p_i\sum_{k=1}^i\frac{(4\alpha)^k}2\\
   &=\sum_{i=1}^{j-1}p_i\frac{(4\alpha)^{i+1}-1}{2(4\alpha-1)}\\
   &\geq\sum_{i=1}^{j-1}\frac{(4\alpha)^{i+1}-1}{2(4\alpha-1)}\\
   &=\frac{1}{2(1-4\alpha)^2} ((4\alpha)^{j+1}-16\alpha^2+(j-1)(1-4\alpha))\\
   &=\frac{1}{2(1-4\alpha)^2}(4\alpha)^{j+1}+O(j).
  \end{align*}
  To turn this lower bound on the length of the path into an lower bound of the number $N$ of steps we use \ref{S:aroundsmooth} and the fact that $\nabla h$ is normal to $\Gamma_k$, so that we have
  \[
   \frac{1}{2(1-4\alpha)^2}(4\alpha)^{j+1}+O(j)\leq\sum_{k=2}^N\frac1{k\log k}\approx\log\log N.
  \]
  Whence 
  \[j\leq A\log\log\log N\]
  for some $A>0$, and \eqref{eq:essaccdef} can be bounded by
  \[
   \frac{\sum_{\substack{x_k\in U\\k\leq N}} \varepsilon_k}{\sum_{i=1}^N\varepsilon_k}\leq \frac{2a/\kappa_U}{(\log N)/L}\sum_{\ell=0}^jp_\ell=O\left(\frac{\sum_{\ell=0}^jp_\ell}{e^{e^j}}\right).
  \]
  Because of the fractal form of the construction of $\Gamma$, we see that the thickness of the tubular neighborhoods around the connected components of $\Gamma_i\cup J$ and those around the connected components around $\Gamma_{i+1}\cup J$ are related by a factor $\alpha$. From our calculation above we conclude that, the number of steps it takes to traverse each $T_j$ increases rapidly, so that in view of \ref{S:thickness}, we see that $p_\ell$ can be uniformly bounded. This means that $\sum_{\ell=0}^jp_\ell\leq Cj$ for some $C>0$, and hence, as $j\to+\infty$,
  \[\frac{\sum_{\ell=0}^jp_\ell}{e^{e^j}}\leq \frac{Cj}{e^{e^j}}\to 0.\]
  This proves that $J\setminus \Gamma$ is not in $\essacc\{x_i\}_i$, and concludes the proof of the lemma.
 \end{proof}

\paragraph{Conclusion.}
 Claim \ref{itex:pathdiff} was proved as item \ref{f2:pathdiff} of Proposition \ref{prop:propertiesf2}. 
 
 It follows from item \ref{f2:crit}  in Proposition \ref{prop:propertiesf2} and Lemma \ref{lem:essacc} that $\Gamma=\essacc\{x_i\}_i\subset\crit f$, and since $f(\Gamma)=[0,1]$, $f$ satisfies claim \ref{itex:nonconstcrit}.
 
 Claim \ref{itex:essacc} is true by item \ref{f2:acc} of Proposition \ref{prop:propertiesf2} and Lemma \ref{lem:essacc}. %
 
 Since $f(\Gamma)=[0,1]$ and $\{x_i\}_i$ bounces endlessly around $\Gamma\cup J$ by item \ref{f2:acc} in Proposition \ref{prop:propertiesf2}, the sequence $\{f(x_i)\}_i$ also does not converge, which is claim \ref{itex:fconv}.
 
 Claim \ref{itex:dim} follows from Lemmas \ref{lem:dimgamma} and \ref{lem:essacc} and item \ref{f2:acc} of Proposition \ref{prop:propertiesf2}.
 
 Claim \ref{itex:slowdown} requires some analysis. Let $x$ and $y$ be distinct points in $J$ with $f(x)>f(y)$. Let $\{x_{i_k}\}_k$ and $\{x_{i_k'}\}_k$ be subsequences that converge to them, respectively, and such that $i_k<i_k'$ for all $k$. Let $u\colon[0,T]\to J$ be a parameterization of $J$ such that $\|u'(t)\|=-(f\circ u)'(t)$ (this determines $T>0$), so that $u$ is a gradient curve, that is, $-u'(t)\in\partial^cf(u(t))$. Then it follows from item \ref{ith2:clarke} in Lemma \ref{lem:defh}, item \ref{f2:clarke} in Proposition \ref{prop:propertiesf2}, and \ref{S:aroundsmooth} that the subgradient sequence $\{x_i\}_i$ goes along $J$ at about the same speed as the neighboring curve $u$, so a very rough estimate of the amount of time it takes for it to go between $x$ and $y$ is $\sup_k\sum_{p=i_k}^{i_k'}\varepsilon_p\leq 2T$, which is claim \ref{itex:slowdown}.
 
 Claim \ref{itex:osccomp} is true because, if we choose the function $Q$ so that its support intersects $J$ but not $\Gamma$, then it follows from item \ref{ith2:clarke} in Lemma \ref{lem:defh}, item \ref{f2:clarke} in Proposition \ref{prop:propertiesf2}, and assumption \ref{S:aroundsmooth} that the averages in the $\liminf$ in \eqref{eq:speedaverages} asymptotically approach 
 \[\left\|\frac{\int_0^TQ(u(t))u'(t)dt}{\int_0^TQ(u(t))dt}\right\|\neq 0,\]
 with $u$ as in our discussion of claim \ref{itex:slowdown} above, which immediately implies inequality \eqref{eq:speedaverages}.
 
 Claim \ref{itex:perp} follows immediately from item \ref{ith2:clarke} in Lemma \ref{lem:defh}, and assumptions \ref{S:closetoTi} and \ref{S:aroundsmooth}.

 \appendix
 \section{Proof of Lemma \ref{lem:defh}}\label{sec:proofh}
 
 For $i>0$ and a connected component $C$ of $J\cup\Gamma_i$, and let $W_C$ be its tubular neighborhood with chart $\varphi_C$.
 
 On the tubular neighborhood $W_C$, we define $h$ by \eqref{eq:defh}.
 Observe that with this definition, $h$ is $C^\infty$ on each of the two connected components of $W_C\setminus C$, settling items \ref{ith2:piecewisesmooth} and \ref{ith2:hessian}. 
 
 Since the coordinates given by the charts $\varphi_C$ are compatible for the different connected components $C$ of the sets $J\cup\Gamma_i$, $i\in\N$, this defines $h$ on the closure of the union $R=\overline{\bigcup_CW_C}$. In particular $\Gamma\subset R$. From \eqref{eq:defh}, we see that $h$ is smooth on each connected component of $R\setminus \Gamma$.

 In the following, we will extend $h$ continuously, so the fact that $h$ coincides with $f$ on $\Gamma\cup J$, item \ref{ith2:extendsf}, will follow from the observation that, as we see from \eqref{eq:defh}, it is true on each of connected component of $J\cup\bigcup_i\Gamma_i$. 
 
 Item \ref{ith2:clarke} also follows directly from \eqref{eq:defh} because the components of elements of the Clarke subdifferential in the $\mathbf t$ and $\mathbf n$ directions coincide with the derivatives in the $x$ and $y$ variables, respectively, since 
 \[\left\|\frac{\partial\varphi_C}{\partial x}(x,0)\right\|=1 \quad\textrm{and}\quad \left\| \frac{\partial\varphi_C}{\partial y}(x,y)\right\|=1 .\]
 
 It follows from item \ref{ith2:clarke} that $C\cap\Gamma\subset\crit h$ for each connected component $C$ of $J\cup\bigcup_i\Gamma_i$, so in order to conclude that $\Gamma\subset\crit h$, item \ref{ith2:crit}, we observe that $\Gamma=\overline{\bigcup_C (\Gamma\cap C)}$ and recall that the graph of the Clarke subdifferential is closed.

 Recall
 
 \begin{lem}[Whitney partition of unity {\cite[Lemma 2.5]{bierstone}}]\label{lem:whitneypartition}
 Let $K$ be a compact subset of $\R^n$. There exists a countable family of functions $\phi_i\in C^\infty(\R^n\setminus K)$, $i\in\N$, such that
 \begin{enumerate}
  \item for each $x\in\R^n\setminus K$ there are at most $3^n$ numbers $i\in\N$ such that $x\in\supp\phi_i$,
  \item $\phi_i\geq 0$ for all $i\in\N$, and $\sum_{i}\phi_i(x)=1$ for all $x\in\R^n\setminus K$,
  \item\label{whit:dist} $2\dist(\supp\phi_i,K)\geq \diam(\supp\phi_i)$ for all $i\in\N$,
  \item \label{it:derivativebound} there exist constants $C_k>0$, depending only on $k$ and $n$, such that, if $x\in\R^n\setminus K$, then
  \[\|D^k\phi_i(x)\|\leq C_k\left(1+\frac{1}{\dist(x,K)^{|k|}}\right)\]
  where $D^kg$ denotes the $k^{\textrm{th}}$ derivative of $g$. 
 \end{enumerate}
\end{lem}

 Let $\{\phi_i\}_i$ be a Whitney partition of unity of $\R^2\setminus R$, as in Lemma \ref{lem:whitneypartition} with $K=R$. 
 For each $i$, choose be a point $p_i$ in $R$ minimizing the distance to $\supp\phi_i$.
 Although the definition \eqref{eq:defh} does not give $h$ on a neighborhood of each $p_i$ not on $\Gamma$, we may assume that $\nabla h(p_i)$ is well defined, perhaps after shrinking $R$ slightly.
 Let $g_i\colon\R^2\to\R$ be the affine function given by 
 \[g_i(x)=\begin{cases}
           h(p_i)+\nabla  h(p_i)\cdot(x-p_i),&p_i\notin \Gamma,\\
            h(p_i),&p_i\in\Gamma.
          \end{cases}
 \]
 Similar to the proof \cite{bierstone} to Whitney's extension theorem, for $x\in\R^2\setminus R$ we define
 \[
  h(x)=\sum_ig_i(x)\phi_i(x).
 \]
 On $\R^2\setminus R$, it is clear that $h$ is smooth, because locally it is  a finite sum of smooth functions.
 
 On the other hand, on the set $(\partial R)\setminus \Gamma$ that is the boundary of $R$ with $\Gamma$ removed, by construction $h$ is differentiable and its gradient is continuous. To see why, one can use the same technique as in the well-known proof of Whitney's extension theorem \cite[Theorem 2.3]{bierstone}; we sketch the main ideas. To show that $h$ is differentiable at $r\in (\partial R)\setminus \Gamma$, it is enough to show that 
 \[
  \left|h(x)-h(r)-\nabla h(r)\cdot (x-r)\right|=O(\|x-r\|^2)%
 \]
 for $x\in\R^2\setminus R$ such that the point $r$ minimizes the distance from $x$ to $R$. For this, one uses the fact that $ h$ is smooth on each connected component of $W_C\setminus(\Gamma\cup J)$, so that for $p_i$ near $r$ we have the Taylor estimates
 \[| h(p_i)+\nabla h(p_i)\cdot (r-p_i)- h(r)|=O(\|r-p_i\|^2)\]
 and
 \[ \|\nabla  h(r)-\nabla h(p_i)\|=O(\|r-p_i\|).\]
 These give
 \begin{align*}
 |h(x)&- h(r)-\nabla  h(r)\cdot (x-r)|\\
  &= \left|\sum_ig_i(x)\phi_i(x)- h(r)-\nabla  h(r)\cdot (x-r)\right|\\
  &\leq \sum_i\phi_i(x)\left|g_i(x)- h(r)-\nabla  h(r)\cdot (x-r)\right|\\
  &=\sum_i\phi_i(x)\left| h(p_i)+\nabla  h (p_i)(x-p_i)- h(r)-\nabla  h(r)\cdot (x-r)\right|\\
  &\leq\sum_i\phi_i(x)(| h(p_i)+\nabla  h(p_i)\cdot (r-p_i)- h(r)|\\
  &\qquad\qquad\qquad+|\nabla h(p_i)\cdot(x-r)-\nabla h(r)\cdot(x-r)|)\\
  &=\sum_i\phi_i(x)(O(\|r-p_i\|^2)+O(\|r-p_i\|\|x-r\|))\\
  &=\sum_i\phi_i(x)O(\|x-r\|^2)=O(\|x-r\|^2),
 \end{align*}
 since, by item \ref{whit:dist} in Lemma \ref{lem:whitneypartition},
 \begin{align*}
  \|x-p_i\|
   &\leq\diam(\supp\phi_i)+\dist(\supp\phi_i,R)\\
   &\leq \diam(\supp\phi_i)+\dist(x,R)\\
   &\leq 2\dist(\supp\phi_i,R)+\dist(x,R)\\
   &\leq3\dist(x,R)=3\|x-r\|
 \end{align*}
 and 
 \[\|r-p_i\|\leq \|r-x\|+\|x-p_i\|\leq 4\|x-r\|.\]
 Similarly, since $\sum_i\nabla\phi_i(x)=0$ because $\sum_i\phi_i(x)=1$, so that also $\sum_i\nabla\phi_i(x)(\nabla h(r)\cdot(x-r))=0$, and using the triangle inequality and Lemma \ref{lem:whitneypartition},
 \begin{align*}
  \|&\nabla h(r)-\nabla h(x)\|=\left\|\nabla h(r)-\nabla\sum_ig_i(x)\phi_i(x)\right\|\\
  &=\left\|\nabla h(r)-\nabla\sum_i\phi_i(x)(h(p_i)+\nabla h(p_i)\cdot(x-p_i))\right\|\\
  &\leq\left\|\sum_i\phi_i(x)(\nabla h(r)-\nabla h(p_i))\right\|+\left\|\sum_i\nabla\phi_i(x)(h(p_i)+\nabla h(p_i)\cdot(x-p_i))\right\|\\
   &=\left\|\sum_i\phi_i(x)(\nabla h(r)-\nabla h(p_i))\right\|+\Bigg\|\sum_i\nabla\phi_i(x)[(h(p_i)
    +\nabla h(p_i)\cdot(r-p_i))\\
   &\hspace{6cm}+(\nabla h(p_i)\cdot(x-r)+\nabla h(r)\cdot(x-r))]\Bigg\|\\
   &\leq\sum_i\phi_i(x)O(\|r-p_i\|)+\sum_i\|\nabla\phi_i(x)\|\big(O(\|r-p_i\|^2)\\
   &\hspace{6cm}+O(\|p_i-r\|)\,O(\|x-r\|)\big)\\
   &\leq O(\|r-p_i\|)+3^2C_1\left(1+\frac1{\|x-r\|}\right)\big(O(\|r-p_i\|^2)\\
   &\hspace{6cm}+O(\|p_i-r\|)\,O(\|x-r\|)\big)\\
   &=O(\|x-r\|)
 \end{align*}
 Thus $h$ is $C^1$ on $\R^2\setminus(\Gamma\cup J)$, which settles item \ref{ith2:C1}. 
 
 It also follows from that, together with the fact that from \eqref{eq:defh} we know that $h$ is Lipschitz with constant $\leq 3L$ on $R$, that $h$ is locally Lipschitz.
 
 To prove that $h$ is path-differentiable, 
 let $\gamma\colon\R\to\R^2$ be a Lipschitz curve. By Lemma \ref{lem:cantornull}, $\gamma^{-1}(\Gamma\setminus\bigcup_i\Gamma_i)$ is a set of measure zero. The set of points $t\in\R$ such that $\gamma(t)\in \bigcup_i\Gamma_i$ with $\gamma'(t)$ not tangent to $\bigcup_i\Gamma_i$ is countable as each such $t$ is isolated. If $\gamma(t)$ is in  $J\cup\Gamma_i$ for some $i$, and $\gamma'(t)$ is tangent to $J\cup \Gamma_i$, then it follows from item \ref{ith2:clarke} that the chain rule condition for path-differentiability holds at $t$, and this condition also holds on $\R^2\setminus (\Gamma\cup J)$ because of item \ref{ith2:C1}. This proves item \ref{ith2:pathdiff}.

\begin{lem}\label{lem:cantornull}
 If $\gamma\colon\R\to\R^2$ is Lipschitz with $\gamma'(t)\neq 0$ for almost every $t$, then $\gamma^{-1}(\Gamma\setminus\bigcup_i\Gamma_i)$ has measure zero.
\end{lem}
\begin{proof}
  Write $\gamma=(\gamma_1,\gamma_2)$ for the two coordinate components of $\gamma$. Let  $P_\ell$ be the projection of $\Gamma\setminus\bigcup_i\Gamma_i$ into the $\ell^\textrm{th}$ coordinate axis, $\ell=1,2$. Note that since $\alpha\leq \frac13$, $P_\ell$ is a Cantor set of measure zero, $|P_\ell|=0$.
  
  Because of Rademacher's theorem and the fact that $\gamma'(t)\neq 0$ for almost every $t\in\R$, the sets $A_1$ and $A_2$ where the derivatives $\gamma_1'(t)$ and $\gamma_2'(t)$, respectively, are well-defined and nonzero, satisfy that $A_1\cup A_2$ is a set of full measure. If $B$ is the null set of real numbers $t\in\R$ such that $\gamma'(t)$ is either not defined or equal to zero, then $A_1\cup A_2\cup B=\R$.
  
  For $i\in\{1,2\}$ and $p\in P_i$, $\gamma_i^{-1}(p)\cap A_i$ is countable because the isolated points in this set are only countably-many, and the non-isolated points $t\in \gamma_i^{-1}(p)\cap A_i$ either satisfy $\gamma_i'(t)=0$ (as can be seen by taking the limit in the definition of the derivative restricting to points in $\gamma_i^{-1}(p)\cap A_i$) or $\gamma_i'(t)$ is not defined; in other words, if $t\in \gamma_i^{-1}(p)\cap A_i$ is not isolated, then $t\in B$.
  
  Thus $\gamma_i^{-1}(P)\cap A_i$ can be written as a countable, disjoint union $\bigsqcup_{j=1}^\infty Q_j^i$ of measurable sets $Q_j^i\subset A_i$ such that $\gamma_i(Q_j^i\cup B)=P_i$ and $\gamma_i$ is injective on $Q^i_j$. %
  Since, by the change of variable formula \cite[p. 99]{evansgariepy},
  \[0\leq\int_{Q_j^i}\|\gamma_i'(t)\|dt=|\gamma_i(Q_j^i)|\leq |P_i|=0,\]
  and since this is only possible if $|Q^i_j|=0$ because the integrand is strictly positive throughout $Q^i_j$,
  all the sets $Q^i_j$ must be Lebesgue null. As a consequence, $\gamma_i^{-1}(P)\cap A_i$ is a countable union of null sets, and it is hence null.
  
  Now, 
   \[
    \gamma^{-1}(\Gamma\setminus\textstyle\bigcup_i\Gamma_i)\displaystyle=\gamma_1^{-1}(P_1)\cap \gamma_2^{-1}(P_2)\subseteq(\gamma_1^{-1}(P_1)\cap A_1)\cup (\gamma_2^{-1}(P_2)\cap A_2)\cup B,
   \]
   and the three sets on the right-hand side are null, so this proves the lemma.
\end{proof}

 \paragraph{Acknowledgements.} 
The author is deeply grateful for the guidance and support of J\'er\^ome Bolte and Edouard Pauwels. The author acknowledges the support of ANR-3IA Artificial and Natural Intelligence Toulouse Institute.

 \bibliography{bib}{}
 \bibliographystyle{plain}

\end{document}